\documentclass[a4paper,12pt]{article}
\usepackage{amsmath,amssymb,amsthm}
\usepackage{bbm}%\mathbbm{}
\usepackage{graphicx}
\usepackage[english]{babel}
\newtheorem{definition}{Definition}
\newtheorem{theorem}{Theorem}
\newtheorem*{theorem1}{Theorem}
\newtheorem{proposition}{Proposition}
\newtheorem{lemma}{Lemma}

\theoremstyle{remark}
\newtheorem{remark}{Remark}
\def\R{\mathbb R}
\title{A Minimality property for entropic solutions to scalar conservation laws in $1+1$ dimensions}
\author{Michael Blaser$^{*}$\;, \;Tristan Rivi\`ere\footnote{Department of Mathematics, ETH Zentrum,
8093 Z\"urich, Switzerland.}}

\def\res{\mathop{\hbox{\vrule height 7pt width .5pt 
depth 0pt\vrule height .5pt width 6pt depth 0pt}}\nolimits}

\begin{document}
\maketitle
{\bf Abstract :} {\it The Second Law of Thermodynamics asserts that the physical entropy of an adiabatic system is an increasing function in time.
In this paper we will study a more stringent version of this law, according to which the entropy should not only increase in time, but the rate of increase is optimal in absolute value among all possible evolutions.
We will establish this property in the framework of non-linear scalar hyperbolic  conservation law with strictly convex fluxes. }

\section{Introduction}
 We consider solutions to the following equation
\begin{equation}\label{scl}
\left .
\begin{array}{rclcl}
 \partial_{t}u+\partial_{x}f(u)& =&0&\mbox{in}&\mathbb{R}_+\times\mathbb{R}\,, \\ 
u(x,0) & =&u_{0}(x)\,,
\end{array}
\right \}
\end{equation}
where the flux $f$ is strictly convex ( $f''\ge c>0$ )and the initial date $u_0\in L^{\infty}$. It is well known, that, even for smooth initial data, the classical solution can cease to exist in finite time, due to the possible formation of shocks (see Chapter 4.2 in \cite{Da}). Therefore one has to consider weak solutions of (\ref{scl}), i.e. solutions, which satisfy (\ref{scl}) in the distributional sense. However it turned out, that, for a given initial data, the space of weak solutions is huge (see Chapter 4.4 in \cite{Da}). Therefore additional conditions have to be imposed to single out the physical relevant weak solutions in some models.

In 1957 Oleinik proved in \cite{Ol} uniqueness of bounded weak solutions, which satisfy almost everywhere her '\textit{E-condition}'
\begin{equation}\label{e_condition}
 u(y,t)-u(x,t)\leq \frac{y-x}{ct}\,,\quad\mbox{for}\quad x<y\,,t>0\,,
\end{equation}
where $c=\inf f''$. A immediate consequence of this condition (\ref{e_condition}) is a spectacular regularization phenomena. Oleinik proved, that for bounded measurable initial data, the weak solution satisfying almost everywhere (\ref{e_condition}) becomes immediately locally BV in space and locally in space-time in the complement of the initial line .

A more powerful approach was given by Kruzhkov in \cite{Kr}, where he replaces condition (\ref{e_condition}) by a family of integral inequalities. This approach covers also cases, where $f$ is non-convex and the space dimension is bigger than one. However in the case of convex fluxes one can show, that his \textit{entropy condition} is equivalent to Oleinik's \textit{E-condition} (see Chapter 8.5 in \cite{Da}). More precisely for $u_0\in L^{\infty}$ he proved existence and uniqueness of weak solutions satisfying the \textit{entropy condition}: He considers the family of convex entropy flux pairs $(\eta_a, \xi_a)_{a\in\mathbb{R}}$, where
\begin{equation}\label{kruzhkov_family}
 \eta_{a}(u)=(u-a)^+\quad\mbox{and}\quad \xi_a(u)=\operatorname{sign}(u-a)^+(f(u)-f(a))\,,
\end{equation}
and $w^+$ stands for $\max\{w, 0\}$. Then an entropy solution is a bounded function $u$, which satisfies (\ref{scl}) in the sense of distributions and
\begin{equation}\label{e_inequality}
\partial_t\eta_a(u)+\partial_x\xi_a(u)\leq0\,.
\end{equation} 
Equivalently one can replace the one parameter family $(\eta_a, \xi_a)_{a\in\mathbb{R}}$ and assume, that (\ref{e_inequality}) is fulfilled for all convex $\eta$ with corresponding entropy flux $\xi$, which is defined by $\xi=\int \eta'f'$.
As a consequence of this one can show, if the initial data $u_0$ is in BV, that $u$ is in BV for all later times.

Let $a\wedge b$ denote $\min\{a,b\}$. Let $u\in L^{\infty}(\mathbb{R}\times[0,T))$ be a weak solution of (\ref{scl}), such that 
\[m(x,t,a)=\partial_{t}(u\wedge a)+\partial_x f(u\wedge a)\in\mathcal{M}_{loc}(\mathbb{R}\times\mathbb{R}_+\times\mathbb{R})\,\]
where $\mathcal{M}$ denotes the space of Radon measures. One can define the absolute value of the entropy production over a set $\Omega\subset\mathbb{R}\times\mathbb{R}_+$ as being
\begin{equation}\label{entropy_production}
 EP=\int_{\mathbb{R}}|m|(\Omega,a)\,da\,.
\end{equation}
In the case of $u$ being an entropy solution and hence in BV, the measure $m(x,t,a)$ and therefore the entropy production of $u$ simplifies to
\begin{equation}\label{ep_entropy_solution}
 EP=\int_{\Omega}\Delta(u^{+},u^{-})\mathcal{H}^1\res J_u\,,
\end{equation}
where $J_u$ denotes the rectifiable set of jump points of $u$, $u_+$ and $u_-$ are respectively the left and right approximate limits of $u$ for some orientation of $J_u$  and
\begin{equation}\label{Delta}
\Delta(a,b)=\frac{(a-b)^2\left[\frac{f(a)+f(b)}{2}\right]-(a-b)\int_{a}^{b}f(s)\,ds}{\left[(a-b)^2+(f(a)-f(b))^2\right]^{\frac12}}\,.
\end{equation}
It is natural to compare the different entropic productions of the weak solutions to (\ref{scl})  - $BV$ or not $BV$ ! - and to ask the following questions : {\it does there exists a weak solution which minimizes the entropy production and, if so, what properties does a minimizer of (\ref{entropy_production}) have}. 

In this work we provide a partial answer to this question. We show  a weak solution of (\ref{scl}) whose entropy production increases in time less, than any other weak solution's entropy production, has to be the entropy solution. Precisely 

Let $\mathcal{W}$ denote the set of defect measures induced by a weak solution of (\ref{scl}), i.e.
\begin{equation}
\mathcal{W}:=\left \{ 
\begin{array}{c}
m(x,t,a)\in \mathcal{M}_{loc} ~\text{s.t.}~m(x,t,a)=\partial_{t}(u\wedge a)+\partial_{x}f(u\wedge a),\\
\text{where}~u\in L^{\infty}~\text{is a weak sol. of (\ref{scl}).}\end{array}\right\}
\end{equation}
Our main result in the present work is the following.
\begin{theorem}\label{minimality_theorem}
Let $f\in C^{2}(\mathbb{R})$ such that $f''\geq c>0$ and
\begin{equation}\label{infinity_behaviour}
 \lim_{|x|\rightarrow\infty}f(x)=\infty\,.
\end{equation}
Moreover let $u_{0}\in L^{\infty}(\mathbb{R})$ be compactly supported. Let $u\in L^{\infty}(\mathbb{R}\times [0,T))$ be an arbitrary weak solution of (\ref{scl}), such that  $m(x,t,a)=\partial_t(u\wedge a)+\partial_x f(u\wedge a)$  is locally a Radon measure in ${\R}\times [0,T)\times {\R}$ . Assume the ''entropy production'' $m$ satisfies
\begin{equation}\label{minimality}
\int_{\mathbb{R}}\left|m\right|\left(\mathbb{R}\times(0,\bar{t}),a\right)\,da\leq \int_{\mathbb{R}}\left|q\right|\left(\mathbb{R}\times(0,\bar{t}),a\right)\,da~\forall~q\in\mathcal{W}~\text{and}~\forall \bar{t}\in(0,T)\,.
\end{equation}
Then $u$ is the entropy solution, i.e. satisfies (\ref{e_condition}) and equivalently (\ref{e_inequality}).
\end{theorem}

A similar criteria in a more restrictive setting is considered by Dafermos in Chapter 9.7 of \cite{Da}. He considers weak solutions $u$ of (\ref{scl}) with initial data
\begin{equation}\label{riemann}
u_0(x)=\left\{\begin{array}{lcl}
                 u_l&\mbox{if}&x<0\,,\\
		u_r&\mbox{if}&x>0\,.
                \end{array}\right.
\end{equation}
Since the conservation law is invariant under Galilean transformations it is reasonable in this case to consider weak solutions of the form\[ u(x,t)=v\left(\frac{x}{t}\right)\,.\]
One can then define $\omega=\frac{x}{t}$ and consider $v$ as a function only dependent of $\omega$, i.e. $v=v(\omega)$. Then $v(\omega)$ satisfies the ordinary differential equation
\[\frac{d}{d\omega}\left(f(v(\omega))-\omega v(\omega)\right)+v(\omega)=0\]
in the sense of distributions and has prescribed end states
\[\lim_{\omega\rightarrow-\infty}v(\omega)=u_l\quad\mbox{and}\quad\lim_{\omega\rightarrow\infty}v(\omega)=u_r\,.\]
Furthermore it is assumed that $v$ is in $BV$ and denotes $J_v$ the set of jump points $\omega$  for $v$. For a given entropy-entropy flux pair $(\eta(u),\xi(u))$ C. Dafermos defines the combined entropy of the shocks in $v$ by
\begin{equation}\label{combined_entropy}
\mathcal{P}_{v}=\sum_{\omega\in J_v}\left\{\xi(v(\omega+))-\xi(v(\omega-))-\omega\left[\eta(v(\omega+))-\eta(v(\omega-))\right]\right\}\,.
\end{equation}
Furthermore he introduces the rate of change of the total entropy production 
\[\dot{\mathcal{H}}_v=\frac{d}{dt}\int_{-\infty}^{\infty}\eta(u(x,t))\,dx=\int_{-\infty}^{\infty}\eta(v(\omega))\,d\omega\,,\]
for entropy-entropy flux pairs $(\eta,\xi)$ such that $\eta(u_l)=\eta(u_r)=0$.

He shows that in this simple case the rate of change of the total entropy and the entropy productions are related to each other by 
\begin{equation*}
 \dot{\mathcal{H}}_v=\mathcal{P}_v+\xi(u_l)-\xi(u_r)\,.
\end{equation*}
We can now relate the combined entropy $\mathcal{P}_v$ to our entropy productions (\ref{entropy_production}). To do so one notices, that for a $T>0$, $\psi\in C^{\infty}_{c}(\mathbb{R}\times(0,T))$ and an entropy-entropy flux pair $(\eta,\xi)$ we get after a change of variable
\begin{equation}\label{relation_entropy_production}
\begin{split}
\int_{\mathbb{R}\times[0,T]}&\eta(u(x,t))\partial_t\psi+\xi(u(x,t))\partial_x\psi\,dx\,dt\\
&=\int_{\mathbb{R}}\xi(v(\omega))\frac{d}{d\omega}\phi(\omega)-\eta(\omega)\frac{d}{d\omega}\left[\omega\cdot\phi(\omega)\right]\,d\omega\\
&=\sum_{\omega\in J_v}\phi(\omega)\left\{\xi(v(\omega+))-\xi(v(\omega-))-\omega\left[\eta(v(\omega+))-\eta(v(\omega-))\right]\right\}\,,
\end{split}
\end{equation}
where
\begin{equation}\label{def_phi}
\phi(\omega)=\int_0^T\psi(\omega t,t)\,dt\,.
\end{equation}
For a jump point $\omega$ we write $v_+=v(\omega+)$ and $v_-=v(\omega-)$, then taking the particular entropy-entropy flux pair $(\eta_a,\xi_a)$, defined in (\ref{kruzhkov_family}) and using identity (\ref{relation_entropy_production}) gives
\begin{multline}\label{relation_entropy_production1}
\int_{\mathbb{R}}\int_{\mathbb{R}\times[0,T]}\psi(x,t)dm(x,t,a)\,da\\
=-\int_{\mathbb{R}}\sum_{\omega\in J_v}\phi(\omega)\left\{\xi_a(v_+)-\xi_a(v_-))-\omega\left[\eta_a(v_+)-\eta_a(v_-)\right]\right\}\,da\\
\end{multline}
A short calculation reveals
\begin{multline}\label{relation_entropy_production2}
\int_{\mathbb{R}}\xi_a(v_+)-\xi_a(v_-)-\omega\left[\eta_a(v_+)-\eta_a(v_-)\right]\,da\\=-\frac12\omega\left(v_-^2-v_+^2\right)-\int_{v_-}^{v_+}sf'(s)\,ds\,,
\end{multline}
where we used the Rankine-Hugoniot condition for self similar solutions:
\[f(v_+)-f(v_-)=\omega(v_+-v_-)\,.\]
Applying (\ref{relation_entropy_production2}) in (\ref{relation_entropy_production1}) gives
\begin{multline}\label{relation_entropy_production3}
\int_{\mathbb{R}}\int_{\mathbb{R}\times[0,T]}\psi(x,t)dm(x,t,a)\,da\\
=\sum_{\omega\in J_v}\phi(\omega)\left\{\xi(v(\omega+))-\xi(v(\omega-))-\omega\left[\eta(v(\omega+))-\eta(v(\omega-))\right]\right\}\,da\,,
\end{multline}
where $(\eta,\xi)$ is the entropy-entropy flux pair
\begin{equation}\label{part_entropy_pair}
\eta(v)=\frac12v^2\quad\mbox{and}\quad \xi(v)=\int_0^v sf'(s)\,ds\,.
\end{equation}
From (\ref{relation_entropy_production3}) we deduce with (\ref{def_phi})
\begin{equation}\label{relation_entropy_production4}
\frac{1}{T}\int_{\mathbb{R}}m(\mathbb R\times[0,T],a)\,da=\mathcal{P}_v\,
\end{equation}
where the combined entropy production $\mathcal{P}_v$ is taken for the entropy-entropy flux pair defined in (\ref{part_entropy_pair}).
Since $T>0$ is arbitrary and $\mathcal{P}_v$ independent of $T$ it follows from (\ref{relation_entropy_production4})
\begin{equation}\label{relation_entropy_production5}
\frac{d}{dt}\int_{\mathbb{R}}m(\mathbb R\times[0,t],a)\,da=\mathcal{P}_v\quad\mbox{for all}\quad t>0\,,
\end{equation}
which finally relates (\ref{combined_entropy}) to (\ref{entropy_production}).

Then a weak solution $u=v\left(\frac{x}{t}\right)$ of (\ref{scl}) with initial data (\ref{riemann}) is said to satisfy the \textit{entropy rate admissibility criterion} if
it satisfies the following optimality criterion of the entropy production 
\begin{equation*}
 \mathcal{P}_v\leq\mathcal{P}_{\tilde{v}}
\end{equation*}
 or equivalently
\begin{equation*}
 \dot{\mathcal{H}}_v\leq\dot{\mathcal{H}}_{\tilde{v}}
\end{equation*}
holds, for any other weak solution $\tilde{u}=\tilde{v}\left(\frac{x}{t}\right)$ of (\ref{scl}) with initial condition (\ref{riemann}).

Using (\ref{relation_entropy_production5}) one can express the entropy rate admissibility criterion for the particular entropy-entropy flux pair in (\ref{part_entropy_pair}) in terms of the entropy production (\ref{entropy_production}): A solution $u=v\left(\frac{x}{t}\right)$ with initial data (\ref{riemann}) and  defect measure $m(x,t,a)$ satisfies \textit{entropy rate admissibility criterion} if
\begin{equation}\label{entropy_rate}
\frac{d}{dt}\int_{\mathbb{R}}m(\mathbb R\times[0,t],a)\,da\leq \frac{d}{dt}\int_{\mathbb{R}}\tilde{m}(\mathbb R\times[0,t],a)\,da\quad\mbox{for all}\quad t>0\,
\end{equation}
for any other weak solution $\tilde{u}=\tilde{v}\left(\frac{x}{t}\right)$ of (\ref{scl}) with initial condition (\ref{riemann}) and defect measure $\tilde{m}(x,t,a)$. One can also integrate (\ref{entropy_rate}) and obtains the equivalent condition
\begin{equation}\label{entropy_rate2}
\int_{\mathbb{R}}m(\mathbb R\times[0,t],a)\,da\leq \int_{\mathbb{R}}\tilde{m}(\mathbb R\times[0,t],a)\,da\quad\mbox{for all}\quad t>0\,.
\end{equation}

Therefore (\ref{entropy_rate}) and (\ref{entropy_rate2}) show, that the entropy rate admissibility criterion can be interpreted as a growth condition of the entropy production (\ref{entropy_production}), which is similar to the growth condition (\ref{minimality}) in Theorem \ref{minimality_theorem}. In Chapter 9.5 of \cite{Da} it is proved:
\begin{theorem1}\cite{Da}
A weak solution $u$ of (\ref{scl}) with initial data (\ref{riemann}) satisfies the entropy rate admissibility criterion for an entropy-entropy flux pair $(\eta,\xi)$ if and only if $u$ satisfies the E-condition (\ref{e_condition}).
\end{theorem1}
Again by (\ref{entropy_rate}) and (\ref{entropy_rate2}) one sees, that this Theorem establishes, similar as in Theorem \ref{minimality_theorem}, a connection between growth rate of the entropy production (\ref{entropy_production}) and entropy admissibility conditions (\ref{e_condition}) and (\ref{e_inequality}).
In Chapter 9.5 there is also an extension of this theorem in the case of strictly hyperbolic systems.

\vspace{0.5cm}

Another results relating an optimality criterion to entropic solution is given by A. Poliakovsky in \cite{Po}. For $u:\mathbb{R}^n\times[0,T]\to\mathbb{R}^k$ he considers a family of energy functionals
\begin{equation}\label{energy1}
I_{\varepsilon, f}(u)=\frac12\int_0^T\int_{\mathbb{R}^n}\left(\varepsilon|\nabla_x u|^2+\frac{1}{\varepsilon}|\nabla_x H|^2\right)\,dx\,dt+\frac12\int_{\mathbb{R}^n}|u(x,T)|^2\,dx\,
\end{equation}
where 
\[\Delta_x H_u=\partial_t u+\operatorname{div}_x f(u)\,.\]
Under certain assumptions on the flux $f$ he shows, that there exists a minimizer to
\[ \inf \left\{ I_{\varepsilon,f}(u):~u(x,0)=u_0(x)\right\} \]
and this minimizer satisfies
\begin{equation*}\left.\begin{array}{rcll}
\partial_t u+\operatorname{div}_x f(u)&=&\varepsilon\Delta_x H_u&\forall(x,t)\in\mathbb{R}^n\times(0,T)\,,\\
u(x,0)&=&u_0(x)&\forall x\in\mathbb{R}^n\,.
\end{array}\right\}
\end{equation*}
In the particular case $k=1$, he calculates the $\Gamma$-limit of (\ref{energy1}) as $\varepsilon\to0^+$ and finds an alternative variational formulation of the admissibility criterion for the
particular solutions to the  scalar conservation laws that can be achieved by this relaxation procedure.

\vspace{0.5cm}

The result of A.Poliakovsky has been inspired by previous works establishing a link between some variational optimality condition of a relaxed problem and the entropy condition at the limit. Among these works we can quote  \cite{RS1}, \cite{RS2} and \cite{ALR}. Let us describe the results established in this 3 works here :

\medskip

We consider for a bounded domain $\Omega\subset\mathbb{R}^2$ the space $\mathcal{M}_{div}(\Omega)$, which consists of unit vectorfields $u$ such that $u=e^{i\varphi}$ for a $\phi\in L^{\infty}(\Omega,\mathbb{R})$ and $\operatorname{div}e^{i\varphi\wedge a}$ is a Radon measure over $\Omega\times {\R}$. This space $\mathcal{M}_{div}$ was introduced by S. Serfaty and the 
second author in \cite{RS1} and \cite{RS2} in connection to a problem related to micromagnetism. We give here a brief description. Let $\Omega$ be a bounded and simply connected domain, for $u\in W^{1,2}(\Omega,\mathbb{S}^1)$ and a $\varepsilon>0$ we consider
\begin{equation}\label{energy}
E_{\varepsilon}(u)=\varepsilon\int_{\Omega}|\nabla u|^2 + \frac{1}{\varepsilon}\int_{\mathbb{R}^2}|H|^2\,,
\end{equation}
where $H=\nabla(G*\hat{u})\,, \hat{u}=u$ on $\Omega$ and $\hat{u}=0$ in $\Omega^c$ and $G$ is the kernel of the Laplacian on $\mathbb{R}^2$.

It was proved in \cite{RS1}, \cite{RS2} that from any sequence $u_{\varepsilon_n}\in W^{1,2}(\Omega,\mathbb{S}^1)$ such that $\varepsilon\to0$ and $E_{\varepsilon_n}(u_{\varepsilon_n})<C$ one can extract a subsequence $u_{\varepsilon_{n'}}$ such that $\varphi_{\varepsilon_{n'}}$ converges strongly in $L^p(\Omega)$ for any $p<\infty$ to a limit $\varphi$ such that $e^{i\varphi}=u\in \mathcal{M}_{div}(\Omega)$.
Furthermore the authors are conjecturing that  the $\Gamma$-Limit should be given by the following functional $E_0$ over ${\mathcal M}_{div}(\Omega)$ :
\begin{equation*}
E_0(u):=2\int_{a\in\mathbb{R}}\left|\operatorname{div}\left(e^{i\varphi\wedge a}\right)\right|(\Omega)\,da
\end{equation*}
Part of the $\Gamma-$convergence has been proved as they established in one hand the following inequality
\begin{equation*}
E_0(u):=2\int_{a\in\mathbb{R}}\left|\operatorname{div}\left(e^{i\varphi\wedge a}\right)\right|(\Omega)\,da\leq\liminf E_{\varepsilon_{n'}}(u_{\varepsilon_{n'}})
\end{equation*}
and in the other hand that
\begin{equation}\label{inf_energy}
\lim_{\varepsilon\to0}\inf_{u\in W^{1,2}} E_{\varepsilon}(u)=2\inf_{u\in\mathcal{M}_{div}(\Omega)}\int_{a\in\mathbb{R}}\left|\operatorname{div}\left(e^{i\varphi\wedge a}\right)\right|(\Omega)\,da=2|\partial\Omega|\,,
\end{equation}
where $|\partial\Omega|$ is the perimeter of the set $\Omega$. One can prove (see \cite{RS1}), that the infimum on the right hand side is achieved by $u=-\nabla^{\perp}\operatorname{dist}(\cdot,\partial\Omega)\in\mathcal{M}_{div}(\Omega)$. The function $g=\nabla^{\perp}\operatorname{dist}(\cdot,\partial\Omega)$ is the viscosity solution of
\begin{equation}\label{eikonal}
\left.\begin{array}{rclcl}
|\nabla g|-1&=&0&\mbox{on}&\Omega\,,\\
g&=&0&\mbox{on}&\partial\Omega\,.
\end{array}\right\}
\end{equation}
A question, which was left open in \cite{RS1} and \cite{RS2} was to describe the possible limits $u$ of minimizing sequence of (\ref{energy}). It was conjectured that  $u=\pm\nabla^{\perp}dist(\cdot,\partial\Omega)$ are the only possible limits of sequences of minimizers.
A positive answer to this conjecture has been given in \cite{ALR}. Precisely, in \cite{RS2} it is proved that the limit $u$ of a minimizing sequence of (\ref{energy}) satisfies 
the {\it entropy condition}
\begin{equation}\label{positive_limit}
\operatorname{div}e^{i\varphi\wedge a}\geq0\quad\mbox{for all}\quad a\in\mathbb{R}
\end{equation}
or $\operatorname{div}e^{i\varphi\wedge a}\leq0\quad\mbox{for all}\quad a\in\mathbb{R}$.
Then in \cite{ALR} the following result is established
\begin{theorem1} \cite{ALR}
Let $u=-\nabla^{\perp}g$ be a divergence free unit vector-field in the space $\mathcal{M}_{div}(\Omega)$. The entropy condition (\ref{positive_limit}) holds if and only if $g$ is a viscosity solution of (\ref{eikonal}) and therefore $g$ is locally semiconcave in $\Omega$ and $u\in BV_{loc}(\Omega,\mathbb{S}^1)$.
\end{theorem1}
Therefore, as a conclusion, one deduces the following equivalences for this particular problem 
\[
\begin{array}{c}
\displaystyle\mbox{viscosity solution to (\ref{eikonal})}\quad\Longleftrightarrow\quad \mbox{ entropy condition (\ref{positive_limit}) }\\[5mm]
\Longleftrightarrow \mbox{ minimality of the entropy production (\ref{inf_energy})}\quad .
\end{array}
\]

\vskip 1cm

The paper is organized as follows: First, in section 2, we establish some technical preliminary results. Then in Section \ref{section_blow_up} we will show, that the measure \[\int_{\mathbb{R}}m(x,t,a)\,da\] has 
 no points with strictly negative density, outside possibly a set of 1-dimensional measure $0$, i.e. we claim
\begin{equation}
\label{uuu}
\lim_{r\rightarrow 0^+}\frac{1}{r}\int_{\mathbb{R}}m\left(B_{r}((x_{0},t_{0})),a\right)\,da\geq 0\quad\mbox{ for }{\mathcal H}^1\mbox{ a.e. }~ (x_{0},t_{0})\in\mathbb{R}\times(0,T).
\end{equation}
In the last section, using an  argument similar to the one used to prove the main result in \cite{ALR}, we deduce that the non negativity condition (\ref{uuu}) implies that $u$ is entropic.

\newpage

%%%%%%%%%%%%%%%%%%%%%%%%%%%%%%%%%%%%%%%%%%%%%%%%%%%%%%%%%%%%%%%%%%%%%%%%%%%%%%%%%%%%%%%%%%%%%%%%%%%%%%%%%%%%%%%%%%%%%%%%%%%%%%%%%%%
%%%%%%%%%%%%%%%%%%%%%%%%%%%%%%%%%%%%%%%%%%%%%%%%%%%%%%%%%%%%%%%%%%%%%%%%%%%%%%%%%%%%%%%%%%%%%%%%%%%%%%%%%%%%%%%%%%%%%%%%%%%%%%%%%%%
%%%%%%%%%%%%%%%%%%%%%%%%%%%%%%%%%%%%%%%%%%%%%%%%%%%%%%%%%%%%%%%%%%%%%%%%%%%%%%%%%%%%%%%%%%%%%%%%%%%%%%%%%%%%%%%%%%%%%%%%%%%%%%%%%%%
\subsection{Preliminary results}\label{preliminary}
In this section we define a notion of weak  entropy solutions (see Definition \ref{def_entropy_solution}) of scalar conservation laws on domain of trapezoidal shape (see (\ref{definition_boundary}). Afterward we will prove Lemma \ref{properties_of_entropy_solutions}, which roughly says that for that kind of entropy solutions the same properties hold as in the classical case. We will use this results in Section 2.2 and Section 2.3.

For $0<t_{1}<t_{2}<T$ and a $\delta>0$ we define the set
\begin{equation}\label{the_set_gamma}
\Gamma^{t_{2}}_{t_{1}}:=\{ (x,t) |~t_{2}> t> \gamma(x,t_{1})\}
\end{equation}
where
\begin{equation}\label{definition_boundary}
\gamma(x,t):=\left\{\begin{array}{ll} t-\hat{\lambda}(x+\delta))&\text{if}~x\leq -\delta\,, \\ 
t&\text{if}~|x|\leq\delta\,,\\ 
t+\hat{\lambda}(x-\delta))&\text{if}~x\geq\delta\,.\end{array}\right.
\end{equation}
for a constant $0<\hat{\lambda}\leq1$. Further we set 
\begin{equation*}
\Lambda_{t_{1}}^{t_{2}}:=\{(x,t)|~(x,t)=(x,\gamma(x,t_{1}))~\text{and}~t_{1}\leq t<t_{2} \}\,.
\end{equation*}

\begin{figure}[h]\label{gamma}
\begin{center}
\includegraphics{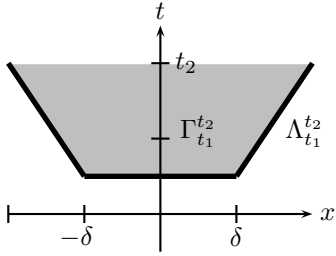}
\end{center}
\caption{The set $\Gamma_{t_1}^{t_2}$}
\end{figure}

As for mentioned we define now a notion of weak respective entropy solution on the domain $\Gamma_{t_{1}}^{t_{2}}$
\begin{definition}\label{def_entropy_solution}
For a $v_{1}\in L^{\infty}(\Lambda_{t_{1}}^{t_{2}})$ we say that $v\in L^{\infty}(\Gamma_{t_{1}}^{t_{2}})$ is weak solution of 
\begin{equation}\label{modified_scl}
 \left.\begin{array}{rclcl}
  \partial_{t}v+\partial_{x}f(x)&=&0&\mbox{in}&\Gamma_{t_{1}}^{t_{2}}\,,\\
v&=&v_{1}&\mbox{on}&\Lambda_{t_{1}}^{t_{2}}\,,
 \end{array}\right\}
\end{equation}
if for all $\psi\in C_{c}^{\infty}(\mathbb{R}\times [0,t_{2}))$
\begin{equation}\label{weak_formulation}
 \int_{\Gamma_{t_{1}}^{t_{2}}}v\partial_{t}\psi+f(v)\partial_{x}\psi\,dx\,dt+\int_{\Lambda_{t_{1}}^{t_{2}}}
\psi\begin{pmatrix} v_1 \\
 -f(v_1)\end{pmatrix}\cdot\tau\,d\sigma=0
\end{equation}
holds, where $\tau$ is the unit tangent vector of $\Lambda_{t_{1}}^{t_{2}}$.
Furthermore we say that $v\in L^{\infty}(\Gamma_{t_{1}}^{t_{2}})$ is an entropy solution of (\ref{modified_scl}), if $v$ additionally satisfies
\begin{equation*}
 q(x,t,a):=\partial_{t}v\wedge a+\partial_{x}f(v\wedge a)\in\mathcal{M}_{loc}\quad\mbox{and}\quad q(x,t,a)\geq0\,.
\end{equation*}
\end{definition}
A priory it is unclear if, for an arbitrary boundary condition $v_{1}\in L^{\infty}(\Lambda_{t_{1}}^{t_{2}})$, the conservation law (\ref{modified_scl}) possess a weak solution or not. We can however prove the following proposition.
\begin{proposition}\label{existence}
 Let $v_1\in L^{\infty}(\mathbb{R}\times[0,T))$ be a weak solution of (\ref{scl}). Then for all $0<\hat{\lambda}\leq1$ and for almost every $t_{1}\in(0,T)$ and all $t_2\in(t_1,T)$ the problem
\begin{equation*}
 \left.\begin{array}{rclcl}
  \partial_{t}v+\partial_{x}f(v)&=&0&\mbox{in}&\Gamma_{t_{1}}^{t_{2}}\,,\\
v&=&v_1&\mbox{on}&\Lambda_{t_{1}}^{t_{2}}\,,
 \end{array}\right\}
\end{equation*}
has an entropy solution in the sense of Definition \ref{def_entropy_solution}.
\end{proposition}
The basic idea for proving Proposition \ref{existence} is to use the correspondence between weak solutions of (\ref{scl}) and viscosity subsolutions of
\begin{equation}\label{hamilton_jacobi}
\left. \begin{array}{rcl}
  \partial_t g+f(\partial_x g)&=&0\,,\\
g(x,0)&=&g_0(x)\,.
 \end{array}\right\}
\end{equation}
Before we are going to prove our assertion, we briefly repeat the definitions of viscosity sub- and supersolutions. We say that $g$ is a viscosity solution of (\ref{hamilton_jacobi}), if for any point $(x_0,t_0)\in\mathbb{R}\times(0,T)$ and for any $\psi \in C^1(\mathbb{R}^2)$ such that $g-\psi$ attains its maximum in $(x_0,t_0)$ the following inequality holds
\[\partial_t \psi(x_0,t_0)+f(\partial_x\psi(x_0,t_0))\leq0\,.\]
Similarly we say, that $g$ is a viscosity supersolution of (\ref{hamilton_jacobi}), if for any point $(x_0,t_0)\in\mathbb{R}\times(0,T)$ and for any for any $\psi \in C^1(\mathbb{R}^2)$ such that $g-\psi$ attains its minimum in $(x_0,t_0)$ the following inequality holds
\[\partial_t \psi(x_0,t_0)+f(\partial_x\psi(x_0,t_0))\geq0\,.\]
We say that $g$ is a viscosity solution of (\ref{hamilton_jacobi}), if $g$ is both a sub- and supersolution. Theorem 2 in \cite{CH} establishes a correspondence between weak solutions of (\ref{scl}) and viscosity subsolutions of (\ref{hamilton_jacobi}).
\begin{theorem}[Conway, Hopf]\label{correspondence}
Let $u\in L^{\infty}(\mathbb{R}\times[0,t))$ be a weak solution of (\ref{scl}). Then there exists a $g\in W^{1,\infty}(\mathbb{R}\times[0,T))$ which satisfies (\ref{hamilton_jacobi}) almost everywhere and is such that $u(x,t)=\partial_x g(x,t)$ and $u_0=\partial_x g(x,0)$ for almost every $x\in\mathbb{R}$.
\end{theorem}
\begin{proof}[\textbf{Proof of Proposition \ref{existence}}]
Let $v_1\in L^{\infty}(\mathbb{R}\times[0,T))$ be a weak solution of (\ref{scl}); then, according to Theorem \ref{correspondence}, there exists  $g_1\in W^{1,\infty}(\mathbb{R}\times[0,T))$, which solves (\ref{hamilton_jacobi}) almost everywhere. By Fubini's Theorem we can choose $t_1$ such that both $\partial_t g_1$ and $\partial_x g_1$ are in $L^\infty(\Lambda_{t_1}^{t_2})$ and
such that
\begin{equation}\label{choice_of_t1}
 \int_{\Lambda_{t_1}^{t_2}}\partial_{t}g_1+f(\partial_x g_1)\,d\sigma=0\quad\mbox{and}\quad v_1=\partial_x g_1\quad\mbox{a.e. on}\quad\Lambda_{t_1}^{t_2}\,.
\end{equation}
For $t_1<t_2<T$ we want to show, that there exists a viscosity solution $g$ of
\begin{equation}\label{viscosity_problem}
 \left.\begin{array}{rclcl}
        \partial_t g+f(\partial_x g)&=&0&\mbox{in}&\Gamma_{t_1}^{t_2}\,,\\
g&=&g_1&\mbox{on}&\Lambda_{t_1}^{t_2}\,.
\end{array}\right\}
\end{equation}
Then we claim, that $v=\partial_x g$ is an entropy solution of (\ref{modified_scl}), in the sense of Definition \ref{def_entropy_solution}. The existence of such a viscosity solution $g$ will be guaranteed by the existence result of Ishi (see Theorem 3.1 in \cite{Is}). In order to be able to apply that theorem we must find a  viscosity subsolution $\underline{g}$ and a viscosity supersolution $\overline{g}$ of (\ref{viscosity_problem}), which satisfy pointwise $\underline{g}=\overline{g}=g_1$ on $\Lambda_{t_1}^{t_2}$ and $\underline{g}\leq\overline{g}$ in $\Gamma_{t_1}^{t_2}$. According to Proposition 5.1 on page 77 in \cite{BC}, the fact that $g_1$ satisfies (\ref{hamilton_jacobi}) almost everywhere implies, that $g_1$ is a viscosity subsolution of (\ref{hamilton_jacobi}). Thus we can put $\underline{g}=g_1$ and it remains to find a viscosity supersolution $\overline{g}$ such that $\overline{g}\geq g_1$ and $\overline{g}=g_1$ on $\Lambda_{t_{1}}^{t_2}$. For two positive constants $A$, $B$ we consider the function
\[ \overline{g}_y(x,t)=g_1(y,\gamma(y,t_1))+A|x-y|+B|t-\gamma(y)|\,.\]
We calculate for $(x,t)\in\Gamma_{t_1}^{t_2}$
\[ \partial_t\overline{g}_y(x,t)+f(\partial_x \overline{g}_y(x,t))=B\operatorname{sign}(t-\gamma(y))+f(a \operatorname{sign}(x-y))\,.\]
By (\ref{infinity_behaviour}) this is positive, if we choose $A$ large enough. Thus
\[\partial_t\overline{g}_y(x,t)+f(\partial_x \overline{g}_y(x,t))>0\quad\mbox{for}\quad (x,t)\in\Gamma_{t_1}^{t_2}\,.\]
Proposition 5.1 on page 77 and Proposition 5.4 on page 78 in \cite{BC} imply, that $\overline{g}$ is a viscosity supersolution. Further we notice, since $g_1\in W^{1,\infty}(\mathbb{R}\times[0,T))$, that for all $y$ and suitable choices of $A$ and $B$
\[g_1(x,t)\leq g(y,\gamma(y))+A|x-y|+B|t-\gamma(y)|\,.\] 
By Proposition 2.11 on page 302 in \cite{BC} 
\[\overline{g}(x,t)=\inf_{y}\overline{g}_y(x,t)\]
is still a supersolution. Furthermore $\overline{w}$ satisfies by construction $\overline{g}=g_1$ on $\Lambda_{t_{1}}^{t_{2}}$ and $\overline{g}\geq g$ in $\Gamma_{t_{1}}^{t_{2}}$. Hence all assumptions of the existence result (Theorem 3.1) in \cite{Is} are fulfilled. Therefore there exists a viscosity solution $g$ of (\ref{viscosity_problem}) such that $g_1\leq g\leq\overline{g}$. By Example 1 in \cite{Is}, the viscosity solution is Lipschitz continuous, i.e. $g\in W^{1,\infty}(\Gamma_{t_1}^{t_2})$. For $(x,t)\in\Gamma_{t_1}^{t_2}$ and $(y,s)\in\Lambda_{t_1}^{t_2}$ we notice
that
\[g_1(x,t)-\overline{g}(y,s)\leq g(x,t)-g(y,s)\leq \overline{g}(x,t)-g_1(y,s)\,.\]
Using the fact that $g_1$ is Lipschitz continuous and the construction of $\overline{g}$ we deduce from the previous line
\[-\|(x,t)-(y,s)\|C_1\leq g(x,t)-g(y,s)\leq C_2\|(x,t)-(y,s)\|\,,\]
which means $g\in W^{1,\infty}(\Gamma_{t_1}^{t_2}\cup \Lambda_{t_1}^{t_2})$. 

Next we are going to show, that $v=\partial_x g$ is a weak solution of (\ref{modified_scl}) in $\Gamma_{t_1}^{t_2}$ in the sense of Definition \ref{def_entropy_solution}. Since $g$ satisfies (\ref{viscosity_problem}) almost everywhere, it follows for a $\psi\in C^{\infty}_c(\mathbb{R}\times[0,t_2))$
\begin{equation}\label{int1}
\int_{\Gamma_{t_1}^{t_2}}\partial_x\psi\partial_{t}g+f(\partial_xg)\partial_{x}\psi\,dx\,dt=0\,.
\end{equation}
We denote the outer unit normal vector of $\Gamma_{t_1}^{t_2}$ by $n$. Integrating (\ref{int1}) twice by parts gives
\begin{equation}\label{int2}\begin{split}
\int_{\Gamma_{t_1}^{t_2}}\partial_x\psi\partial_{t}g\,dx\,dt&=\int_{\partial\Gamma_{t_1}^{t_2}}g\begin{pmatrix}-\partial_t\psi\\\partial_x\psi\end{pmatrix}\cdot n\,d\sigma+ \int_{\Gamma_{t_1}^{t_2}}\partial_t\psi\partial_{x}g\,dx\,dt\,.
\end{split}\end{equation}
Rewriting the boundary term in (\ref{int2}) and using the fact that $\psi(x,t_2)=0$ leads to
\begin{multline}\label{int3}
\int_{\partial\Gamma_{t_1}^{t_2}}g\begin{pmatrix}-\partial_t\psi\\\partial_x\psi\end{pmatrix}\cdot n\,d\sigma=\int_{\Lambda_{t_1}^{t_2}}g_1\begin{pmatrix}-\partial_t\psi\\\partial_x\psi\end{pmatrix}\cdot n\,d\sigma \\=\int_{s_1}^{s_2} g_1(s,\gamma(s,t_1))\left[\partial_t\psi(s,\gamma(s,t_1))\gamma'(s,t_1)+\partial_x\psi(s,\gamma(s,t_{1}))\right]\,ds\,.
\end{multline}
Integrating the right-hand side of (\ref{int3}) by parts leads to 
\begin{multline}\label{int4}
 \int_{s_1}^{s_2} g_1(s,\gamma(s,t_1))\left[\partial_t\psi(s,\gamma(s,t_1))\partial\gamma(s,t_1)+\partial_x\psi(s,\gamma(s,t_{1}))\right]\,ds\\
=\int_{s_1}^{s_2} g_1(s,\gamma(s,t_1))\frac{d}{ds}\psi(s,\gamma(s,t_1))\,ds=-\int_{s_1}^{s_2}\frac{d}{ds}g_1(s,\gamma(s,t_1))\cdot\psi\,ds\,.
\end{multline}
Therefore, combining (\ref{int3}) and (\ref{int4}) we can rewrite the boundary term in (\ref{int2})
\begin{equation}\label{int5}
\int_{\partial\Gamma_{t_1}^{t_2}}g_1\begin{pmatrix}-\partial_t\psi\\\partial_x\phi\end{pmatrix}\cdot n\,d\sigma=\int_{s_1}^{s_2}\left[\partial_x g_1+\partial_{t}g_1\cdot \partial_s\gamma(s,t_1)\right]\psi\,ds
\end{equation}
Using (\ref{choice_of_t1}), the right-hand side of (\ref{int5}) simplifies to
\begin{align*}
\int_{\partial\Gamma_{t_1}^{t_2}}g_1\begin{pmatrix}-\partial_t\psi\\\partial_x\psi\end{pmatrix}\cdot n\,d\sigma&=\int_{s_1}^{s_2}\left[\partial_x g_1-f(\partial_x g_1)\cdot \partial_s\gamma(s,t_1)\right]\psi\,ds\\
&=\int_{\Lambda_{t_1}^{t_2}}\psi\begin{pmatrix}\partial_x g_1\\-f(\partial_x g_1)\end{pmatrix}\cdot \tau\,d\sigma\,,
\end{align*}
where $\tau$ is the unit tangent vector of $\Lambda_{t_1}^{t_2}$. We replace now the boundary term in (\ref{int2}) using the above identity
\[\int_{\Gamma_{t_1}^{t_2}}\partial_x\psi\partial_{t}g\,dx\,dt=\int_{\Lambda_{t_1}^{t_2}}\psi\begin{pmatrix}\partial_x g_1\\-f(\partial_x g_1)\end{pmatrix}\cdot \tau\,d\sigma+ \int_{\Gamma_{t_1}^{t_2}}\partial_t\psi\partial_{x}g\,dx\,dt\,.\]
Finally, this together with (\ref{int1}) gives
\[\int_{\Lambda_{t_1}^{t_2}}\psi\begin{pmatrix}\partial_x g_1\\-f(\partial_x g_1)\end{pmatrix}\cdot \tau\,d\sigma+ \int_{\Gamma_{t_1}^{t_2}}\partial_t\psi\partial_{x}g+f(\partial_xg)\partial_{x}\psi\,dx\,dt=0\,.\]
Since $v_1=\partial_x g_1$ and by putting $v=\partial_x g$, we see, that $v$ is a solution of (\ref{modified_scl}) in the sense of Definition \ref{def_entropy_solution}.
Finally it remains to show, that $v$ is an entropy solution in the sense, that \[\partial_tv\wedge a+\partial_xf(v\wedge a)\geq0\,.\]
By Corollary 1.7.2 in \cite{CS} $v$ satisfies for all $(x,t),~(y,t)\in\Gamma_{t_1}^{t_2}$ such that $x<y$ 
\[v(y,t)-v(x,t)\leq \frac{y-x}{ct}\,.\]
This immediately implies $q(x,t,a)\geq0$ (see Section 8.5 in \cite{Da}).
\end{proof}

Proposition \ref{existence} being proved, we now establish some properties for entropy solutions to (\ref{modified_scl}) analogous to those in the classical case (see \cite{Da}). Precisely we are going to show
\begin{lemma}\label{properties_of_entropy_solutions}
Let $v_1\in L^{\infty}(\mathbb{R}\times(0,T))$ be a weak solution of (\ref{scl}). Then there exists a constant $\lambda_0>0$, depending on $f$ and $\|v_1\|_{\infty}$, such that, for any domain $\Gamma_{t_1}^{t_2}$ satisfying $0<\hat{\lambda}\leq \lambda_0$, the entropy solution $v\in L^{\infty}(\Gamma_{t_{1}}^{t_{2}})$ of (\ref{modified_scl}) with boundary condition $v_{1}\in L^{\infty}(\Lambda_{t_{1}}^{t_{2}})$ satisfies
\begin{equation}\label{boundary_asymptotic}
\lim_{\varepsilon\to0^+}\int_{s_1}^{s_2-\varepsilon}\left|v(s,\gamma(s,t_1+\varepsilon))-v_1(s,\gamma(s,t_1))\right|\,ds=0\,,
\end{equation}
where \[s_1=-\frac{t_2-t_1}{\hat{\lambda}}-\delta\quad\mbox{and}\quad s_2=\frac{t_2-t_1}{\hat{\lambda}}+\delta\,.\]
Moreover
\begin{equation}\label{Linfty_bound}
             \|v\|_{\infty}\leq \|v_{1}\|_{\infty}\,
\end{equation}
and there exists a constant $C>0$, depending only on $\|v\|_1$ and $\hat{\lambda}$, such that
\begin{equation}\label{control_of_the_defect_measure}
              \int_{\Gamma_{t_{1}}^{t_{2}}}q(x,t,a)\, da\, dx\, dt\leq C(t_{2}-t_{1})\,.
\end{equation}
Let now $w_1,\,w_2\in L^{\infty}(\mathbb{R}\times(0,T))$ be weak solutions of (\ref{scl}). Then there there exists a constant $\lambda_1>0$ depending on $f$ and $\max\{\|w_1\|_{\infty},\|w_2\|_{\infty}\}$ such that, for any domain $\Gamma_{t_1}^{t_2}$ satisfying $0<\hat{\lambda}\leq \lambda_1$ and any choice of two entropy solutions respectively $v_1\in L^{\infty}(\Gamma_{t_{1}}^{t_{2}})$ with boundary condition $w_{1}\in L^{\infty}(\Lambda_{t_{1}}^{t_{2}})$ and $v_2\in L^{\infty}(\Gamma_{t_{1}}^{t_{2}})$  with boundary condition $w_2\in L^{\infty}(\Lambda_{t_{1}}^{t_{2}})$ 
the following holds :  for any $t\in(t_1,t_2)$ and a constant $C>0$ depending on $\Gamma_{t_1}^{t_2}$ and $\max\{\|w_1\|_{\infty},\|w_2\|_{\infty}\}$:
\begin{equation}\label{contraction_inequality}
 \int_{\theta^{-}(t)}^{\theta^{+}(t)}|v_1(x,t)-v_2(x,t)|\, dx\leq C\int_{\Lambda_{t_{1}}^{t_{2}}}|w_1-w_2|\, d\sigma\,,
\end{equation}
where \[\theta^{\pm}(t)=\pm\frac{t-t_1}{\hat{\lambda}}\pm\delta\,.\]
\end{lemma}
\begin{remark}
Inequality (\ref{contraction_inequality}) implies in particular the uniqueness of the entropy solution for a given initial data $w$ on $\Lambda_{t_1}^{t_2}$ issued from a weak solution to (\ref{scl}).
\end{remark}
\begin{proof}[\textbf{Proof of Lemma \ref{properties_of_entropy_solutions}}]
We start to prove (\ref{boundary_asymptotic}). Let $R>0$ such that
\[R+f(\pm R)\geq0\,.\]
We choose $\lambda_0$ such that
\begin{equation}\label{def_lambda0}
\lambda_0^{-1}=\max\left\{\left|f'(R+1+\|v_1\|_{\infty})\right|,\left|f'(-R-1-\|v_1\|_{\infty})\right|\right\}\,.
\end{equation}
We consider now a domain $\Gamma_{t_1}^{t_2}$ such that $\hat{\lambda}\leq\lambda_0$ and an entropy solution $v\in L^{\infty}(\Gamma_{t_1}^{t_2})$ of (\ref{modified_scl}) exists.
From Example 1 in \cite{Is}, we know, that
\begin{equation}\label{bound1}
\|v\|_{\infty}\leq R+1\,.
\end{equation}
Let $\psi\in C^{\infty}_c(\mathbb{R}\times[0,t_2))$. From Theorem 1.3.4 in \cite{Da} we get for all $\varepsilon>0$ and sufficiently small ($t_1+\varepsilon<t_2$)
\begin{equation*}
\int_{\Gamma_{t_1+\varepsilon}^{t_2-\varepsilon}}v\partial_t\psi+f(v)\partial_x\psi\,dx\,dt=-\int_{\partial\Gamma_{t_1+\varepsilon}^{t_2-\varepsilon}}\begin{pmatrix}v\\ -f(v)\end{pmatrix}\cdot\tau\psi\,d\sigma\,.
\end{equation*}
Since $\psi(x,t_2)=0$ this implies
\begin{equation}\label{bound_asymp1}
\int_{\Gamma_{t_1+\varepsilon}^{t_2-\varepsilon}}v\partial_t\psi+f(v)\partial_x\psi\,dx\,dt=-\int_{\Lambda_{t_1+\varepsilon}^{t_2-\varepsilon}}\begin{pmatrix}v\\ -f(v)\end{pmatrix}\cdot\tau\psi\,d\sigma\,.
\end{equation}
As $\varepsilon\to0^+$ the left-hand side of (\ref{bound_asymp1}) converges to
\[\int_{\Gamma_{t_1}^{t_2}}v\partial_t\psi+f(v)\partial_x\psi\,dx\,dt\,.\]
Since $v$ is a weak solution of (\ref{modified_scl}) this later fact implies for the right-hand side of (\ref{bound_asymp1})
\begin{equation}\label{bound_asymp2}
\lim_{\varepsilon\to0^+}\int_{\Lambda_{t_1+\varepsilon}^{t_2-\varepsilon}}\begin{pmatrix}v\\ -f(v)\end{pmatrix}\cdot\tau\psi\,d\sigma=\int_{\Lambda_{t_1}^{t_2}}\begin{pmatrix}v_1\\ -f(v_1)\end{pmatrix}\cdot\tau\psi\,d\sigma
\end{equation}
In order to keep the notation simple we introduce
\begin{equation*}
\bar{\gamma}(s)=\begin{pmatrix}s\\\gamma(s,t_1)\end{pmatrix}\quad\mbox{and}\quad v_{\varepsilon}(x,t)=v(x,t+\varepsilon)\,.
\end{equation*}
From (\ref{bound_asymp2}) we deduce
\begin{equation}\label{bound_asymp3}
\lim_{\varepsilon\to0^+}\int_{s_1+\varepsilon}^{s_2-\varepsilon}\left\{v_{1}(\bar{\gamma}(s))-v_{\varepsilon}(\bar{\gamma}(s))+\hat{\lambda}\left[f(v_{\varepsilon}(\bar{\gamma}(s)))-f(v_1(\bar{\gamma}(s))) \right]\right\}\psi\,ds\,.
\end{equation}
By (\ref{def_lambda0}) we obtain the existence of some constants $C,\,c>0$ for which the following holds
\[ c\leq 1\pm f(\alpha)\leq C\quad\mbox{for all}\quad\alpha\in\left(-R-1-\|v_1\|_{\infty},R+1+\|v\|_{\infty}\right)\,.\]
Therefore we get from (\ref{bound_asymp3}), that
\begin{equation}\label{bound_asymp4}
\lim_{\varepsilon\to0^+}v(s,\gamma(s,t_1+\varepsilon))=v_1(s,\gamma(s))\quad\mbox{for a.e.}\quad s\in[s_1,s_2]\,.
\end{equation}
By dominated convergence, we deduce the claim (\ref{boundary_asymptotic}).

To prove the remaining claims of our lemma, we need to introduce the kinetic formulation of conservation laws, we recommend the introduction to this subject given in \cite{Pe}. However we need here a slight modified version of this formulation. We define for any $v\in\mathbb{R}$
\begin{equation*}
 \chi(v;a):=\mathbbm{1}_{a\leq v}\, ,
\end{equation*}
where $\mathbbm{1}_{a\leq v}$ is the characteristic function of the set $\{a\in\mathbb{R}:~a\leq v\}$. Then a weak solution $v\in L^{\infty}(\Gamma_{t_{1}}^{t_{2}})$ of (\ref{modified_scl}) satisfies in the distributional sense 
\begin{equation}\label{kinetic_formulation}
\left.\begin{array}{rclcl}
\partial_{t}\chi(v(x,t);a)+f'(a)\partial_{x}\chi(v(x,t);a)&=&\partial_{a}q(x,t,a)&\mbox{in}&\Gamma_{t_1}^{t_2}\,,\\
\chi(v;a)&=&\chi(v_1;a)&\mbox{on}&\Lambda_{t_1}^{t_2}\,.
\end{array}\right\}
\end{equation}
In other words this means, that for all $\psi\in C_c^{\infty}(\mathbb{R}\times[0,t_2)\times\mathbb{R})$
\begin{multline}
 \int_{\Gamma_{t_1}^{t_2}}\int_{\mathbb{R}}\chi(v;a)\partial_t\psi+f'(a)\chi(v;a)\partial_x\psi\,da\,dx\,dt\\=\int_{\Gamma_{t_1}^{t_2}}+\int_{\Lambda_{t_1}^{t_2}}\int_{\mathbb{R}}\psi\begin{pmatrix}\chi(v1;a)\\-f'(a)\chi(v_1;a)\end{pmatrix}\cdot\tau\,da\,d\sigma\,.
\end{multline}
In order to prove (\ref{contraction_inequality}), (\ref{Linfty_bound}) and (\ref{control_of_the_defect_measure}) we need to regularize our kinetic equation (\ref{kinetic_formulation}). We choose $\varphi_{1}(x),~\varphi_{2}(t)\in C^{\infty}_{c}(\mathbb{R})$ non-negative functions such that
\[\mbox{supp}~\varphi_1\subset(-1,1)\,,\quad\mbox{supp}~\varphi_{2}\subset[-1,0]\]and\[\int_{\mathbb{R}}\varphi_{2}\,dx=\int_{\mathbb{R}}\varphi_{1}\,dx=1\,.\]
We define the kernel
\begin{equation}\label{kernel}
\varphi_{\varepsilon}(x,t)=\frac{1}{\varepsilon^2}\varphi_1\left(\frac{x}{\varepsilon}\right)\varphi_2\left(\frac{t}{\varepsilon}\right)\,.
\end{equation}
For a constant $C$ depending only from $\hat{\lambda}$ we have 
\begin{equation*}
\operatorname{dist}((x,t),\partial\Gamma_{t_1}^{t_2})>\varepsilon\quad\mbox{for all}\quad(x,t)\in\Gamma_{t_1+C\varepsilon}^{t_2-C\varepsilon}\,.
\end{equation*}
Consequently for $(x,t)\in\Gamma_{t_{1}+C\varepsilon}^{t_{2}-C\varepsilon}$
\begin{equation}\label{zero_on_boundary}
 \varphi_{\varepsilon}(x-y,t-s)=0\quad\mbox{for}\quad(y,s)\in\partial\Gamma_{t_{1}}^{t_{2}}\,.
\end{equation}
We define moreover the two mollified functions
\begin{equation*}
 \chi_{\varepsilon}(x,t,a)=\int_{\Gamma_{t_{1}}^{t_{2}}}\varphi_{\varepsilon}(x-y,t-s)\chi(v(y,s);a)\,dy\,ds
\end{equation*}
and 
\begin{equation*}
 q_{\varepsilon}(x,t,a)=\int_{\Gamma_{t_{1}}^{t_{2}}}\varphi_{\varepsilon}(x-y,t-s)q(y,s,a)\,dy\,ds\,.
\end{equation*}
For $q_{\varepsilon}$ and $(x,t)\in\Gamma_{t_1+C\varepsilon}^{t_2-C\varepsilon}$ we compute
\begin{align*}
q_{\varepsilon}(x,t,a)- q_{\varepsilon}(x,t,b)&=\int_{\Gamma_{t_{1}}^{t_{2}}}\varphi_{\varepsilon}(x-y,t-s)\left[q(y,s,a)-q(y,s,b)\right]\,dy\,ds\\
&=\int_{\Gamma_{t_{1}}^{t_{2}}}\partial_{t}\varphi_{\varepsilon}(x-y,t-s)\left[v\wedge a-v\wedge b\right]\\
&\quad+\partial_{x}\varphi_{\varepsilon}(x-y,t-s)\left[f(v\wedge a)-f(v\wedge b)\right]\,dy\,ds\,,\\
\end{align*}
where we have made use of (\ref{zero_on_boundary}).
Since
\begin{equation*}
 |v\wedge a-v\wedge b|\leq\|v\|_{\infty}|b-a|
\end{equation*}
it follows from the calculation above
\begin{align*}
 \left|q_{\varepsilon}(x,t,a)- q_{\varepsilon}(x,t,b)\right|&\leq \int_{\Gamma_{t_{1}}^{t_{2}}}\left|\partial_{t}\varphi_{\varepsilon}\right|\cdot\|v\|_{\infty}|b-a|+C\left|\partial_{x}\varphi_\varepsilon\right|\cdot\|v\|_{\infty}|b-a|\,dy\,ds\\
&\leq C|b-a|\,.
\end{align*}
Therefore $q_{\varepsilon}$ is Lipschitz continuous with respect to the kinetic variable $a$ and we have for almost every $a\in\mathbb{R}$ in the classical sense
\begin{equation}\label{regularized_equation}
\partial_{t}\chi_{\varepsilon}+f'(a)\partial_{x}\chi_{\varepsilon}=\partial_{a}q_{\varepsilon}(x,t,a)\quad\mbox{in}\quad\Gamma_{t_1+C\varepsilon}^{t_2-C\varepsilon}\,.
\end{equation}
Notice that due to the convolution with $\varphi_{\varepsilon}$ both $\chi_{\varepsilon}$ and $q_{\varepsilon}$ are smooth with respect to $(x,t)$.
Furthermore for $(x,t)\in\Gamma_{t_1+C\varepsilon}^{t_2-C\varepsilon}$ the function $q_{\varepsilon}$ satisfies
\begin{equation}\label{compact_support_of_q}
 q_{\varepsilon}(x,t,a)=0\quad\mbox{if}\quad |a|\geq \|v\|_{\infty}\,.
\end{equation}
This follows from the classical fact, that
\[q(x,t,a)=0\quad\mbox{for}\quad |a|\geq \|v\|_{\infty}\,.\]
Indeed for $|a|\geq \|v\|_{\infty}$ and $\psi\in C^{\infty}_c(\Gamma_{t_1}^{t_2})$ we compute
\begin{align*}
\int_{\Gamma_{t_{1}}^{t_{2}}}q(x,t,a)\psi(x,t)\,dx\,dt&=\int_{\Gamma_{t_{1}}^{t_{2}}}\left[\partial_{t}v(x,t)\wedge a+\partial_{x}f(v(x,t)\wedge a)\right]\psi(x,t)\,dy\,ds\\
&=\int_{\Gamma_{t_{1}}^{t_{2}}}\left[\partial_{t}v(x,t)+\partial_{x}f(v(x,t))\right]\psi(x,t)\,dx\,dt\\
&=-\int_{\Gamma_{t_{1}}^{t_{2}}}v(x,t)\partial_{t}\psi(x,t)+f(v(x,t))\partial_{x}\psi(x,t)\,dx\,dt\\
&=0\,.
\end{align*}
Consider now a convex function $\eta(a)$ in $C^1$, which satisfies
 \[\lim_{a\rightarrow-\infty}\eta(a)=0\]
and denote \[\xi(a)=\int \eta'(a)f'(a)\,da\,.\] We claim that for all $(x,t)\in\Gamma_{t_1+C\varepsilon}^{t_2-C\varepsilon}$ the following holds
\begin{equation}\label{entropy_equality}
 \partial_t \eta_{\varepsilon}(x,t)+\partial_x \xi_{\varepsilon}(x,t)=-\int_{\Gamma_{t_1}^{t_2}}\eta''(a)q_{\varepsilon}(x,t,a)\,da\,,
\end{equation}
where
\[\eta_{\varepsilon}(x,t)=\int_{\Gamma_{t_{1}}^{t_{2}}}\eta(v(y,s))\varphi_{\varepsilon}\left(x-y,t-s\right)\,dy\,ds\]
and
\[\xi_{\varepsilon}(x,t)=\int_{\Gamma_{t_{1}}^{t_{2}}}\xi(v(y,s))\varphi_{\varepsilon}\left(x-y,t-s\right)\,dy\,ds\,.\]
Later will make special choices of $\eta$ in order to get (\ref{Linfty_bound}) and (\ref{control_of_the_defect_measure}).

\medskip

\noindent\textbf{Proof of claim (\ref{entropy_equality})}. We multiply (\ref{regularized_equation}) by $\eta'(a)$
\begin{equation*}
\eta'(a) \partial_{t}\chi_{\varepsilon}+\eta'(a)f'(a)\partial_{x}\chi_{\varepsilon}=\partial_{a}q_{\varepsilon}(x,t,a)\,.
\end{equation*}
Then integrating this equation with respect to $a$ gives
\begin{equation}\label{zwischen1}
 \int_{\mathbb{R}}\eta'(a)\partial_{t} \chi_{\varepsilon}+\eta'(a)f'(a)\partial_{x}\chi_{\varepsilon}\,da= \int_{\mathbb{R}}\eta'(a)\partial_{a}q_{\varepsilon}(x,t,a)\,da\,.
\end{equation}
We compute for the left-hand side
\begin{align*}
\int_{\mathbb{R}}\eta'(a)\chi_{\varepsilon}\,da&=\int_{\Gamma_{t_{1}}^{t_{2}}}\int_{\mathbb{R}}\eta'(a)\chi(v(y,s);a)\varphi_{\varepsilon}\left(x-y,t-s\right)\,da\,dy\,ds\\
&=\int_{\Gamma_{t_{1}}^{t_{2}}}\eta(v(y,s))\varphi_{\varepsilon}\left(x-y,t-s\right)\,dy\,ds=\eta_{\varepsilon}(x,t)
\end{align*}
and similarly for the second term
\begin{equation*}
\int_{\mathbb{R}}\eta'(a)f'(a)\chi_{\varepsilon}\,da=\int_{\Gamma_{t_{1}}^{t_{2}}}\xi(v(y,s))\varphi_{\varepsilon}\left(x-y,t-s\right)\,dy\,ds=\xi_{\varepsilon}(x,t)\,.
\end{equation*}
Thus (\ref{zwischen1}) reduces to
\begin{equation}\label{zwischen2}
\partial_{t}\eta_{\varepsilon}(x,t)+\partial_{x}\xi_{\varepsilon}(x,t)=\int_{\mathbb{R}}\eta'(a)\partial_{a}q_{\varepsilon}(x,t,a)\,da\,.
\end{equation}
Integrating the right-hand side by parts gives
\[\int_{\mathbb{R}}\eta'(a)\partial_{a}q_{\varepsilon}(x,t,a)\,da=-\int_{\mathbb{R}}\eta''(a)q(x,t,a)\,da\,,\]
where we have used the fact, that $q_{\varepsilon}$ is compactly supported in $a$. This gives the result (\ref{entropy_equality}) as claimed.

\medskip

Next we integrate inequality (\ref{entropy_equality}) over the set $\Gamma_{t_1+C\varepsilon}^{\bar{t}}$, where $\bar{t}\in(t_1+C\varepsilon,t_2-C\varepsilon)$. We will abbreviate $t_1+C\varepsilon$ by $\bar{t}_1$. We have
\begin{equation}\label{inequality_for_contraction}
   \int_{\Gamma_{\bar{t}_1}^{\bar{t}}}\partial_{t}\eta_{\varepsilon}(x,t)+\partial_{x}\xi_{\varepsilon}(x,t)\,dx\,dt=- \int_{\Gamma_{\bar{t}_1}^{\bar{t}}}\int_{\mathbb{R}}\eta''(a)q(x,t,a)\,da\,.
\end{equation}
For the first term on the left-hand side of (\ref{inequality_for_contraction}) we compute
\begin{align*}
\int_{\Gamma_{\bar{t}_1}^{\bar{t}}}\partial_{t}\eta_{\varepsilon}\,dx\,dt&=\int_{-\delta}^{\delta}\eta_{\varepsilon}(x,\bar{t})-\eta_{\varepsilon}(x,t_{1})\,dx\\
&\quad+\int_{\delta}^{\frac{(\bar{t}-\bar{t}_{1})}{\hat{\lambda}}+\delta}\int_{\hat{\lambda}(x-\delta)+\bar{t}_{1}}^{\bar{t}}\partial_{t}\eta_{\varepsilon}(x,t)\,dt\,dx\\
&\quad+\int^{-\delta}_{-\frac{\bar{t}-\bar{t}_{1}}{\hat{\lambda}}-\delta}\int_{-\hat{\lambda}(x+\delta)+\bar{t}_{1}}^{\bar{t}}\partial_{t}\eta_{\varepsilon}(x,t)\,dt\,dx\,.
\end{align*}
This gives
\begin{align*}
\int_{\Gamma_{t_1}^{\bar{t}}}\partial_{t}\eta_{\varepsilon}\,dx\,dt&=\int_{-\delta}^{\delta}\eta_{\varepsilon}(x,\bar{t})-\eta_{\varepsilon}(x,\bar{t}_{1})\,dx\\
&\quad+\int_{\delta}^{\frac{\bar{t}-\bar{t}_{1}}{\hat{\lambda}}+\delta}\eta_{\varepsilon}(x,\bar{t})-\eta_{\varepsilon}\left(x,\hat{\lambda}(x-\delta)+\bar{t}_{1}\right)\,dx\\
&\quad+\int^{-\delta}_{-\frac{\bar{t}-\bar{t}_{1}}{\hat{\lambda}}-\delta}\eta_{\varepsilon}(x,\bar{t})-\eta_{\varepsilon}\left(x,-\hat{\lambda}(x-\delta)+\bar{t}_{1}\right)\,dx\,.
\end{align*}
A regrouping of the terms together with a change of variable leads to
\begin{equation}\label{zwischen3}\begin{split}
\int_{\Gamma_{t_1}^{\bar{t}}}\partial_{t}\eta_{\varepsilon}\,dx\,dt&=\int_{\theta^{-}_{\varepsilon}(\bar{t})}^{\theta^{+}_{\varepsilon}(\bar{t})}\eta_{\varepsilon}(x,\bar{t})\,dx-\int_{-\delta}^{\delta}\eta_{\varepsilon}(x,\bar{t}_{1})\,dx\\
&\quad+\hat{\lambda}\int_{\bar{t}_{1}}^{\bar{t}}\eta_{\varepsilon}(\theta^{-}_{\varepsilon}(t),t)-\eta_{\varepsilon}(\theta^{+}_{\varepsilon}(t),t)\,dt\,,
\end{split}\end{equation}
where
\begin{equation}\label{theta}
\theta_{\varepsilon}^{\pm}(t)=\pm\frac{\bar{t}-\bar{t}_1}{\hat{\lambda}}\pm\delta\,.
\end{equation}
Integrating now the second term on the left-hand side of (\ref{inequality_for_contraction}) gives
\begin{equation}\label{zwischen4}
\int_{\Gamma_{\bar{t}_1}^{\bar{t}}}\partial_{x}\xi_{\varepsilon}(x,t)\,dx\,dt=\int_{t_{1}}^{\bar{t}}\xi_{\varepsilon}(\theta_{\varepsilon}^{+}(t),t)-\xi_{\varepsilon}(\theta_{\varepsilon}^{-}(t),t)\,dt\,.
\end{equation}
Inserting (\ref{zwischen3}) and (\ref{zwischen4}) back in (\ref{inequality_for_contraction}) leads to the identity
\begin{equation}\label{main_equality_lemma}\begin{split}
\int_{\theta_{\varepsilon}^{-}(\bar{t})}^{\theta^{+}_{\varepsilon}(\bar{t})}\eta_{\varepsilon}(x,\bar{t})\,dx&=\int_{\Lambda_{\bar{t}_1}^{\bar{t}}}\begin{pmatrix}\eta_{\varepsilon}\\-\xi_{\varepsilon}\end{pmatrix}\cdot\tau\,d\sigma- \int_{\Gamma_{\bar{t}_1}^{\bar{t}}}\int_{\mathbb{R}}\eta''(a)q(x,t,a)\,da\,.
\end{split}
\end{equation}
For suitable choices of $\eta$ this equality (\ref{main_equality_lemma}) will imply the first two claims of Lemma \ref{properties_of_entropy_solutions}.

First we prove (\ref{Linfty_bound}). Let $a_0$ be a real number being fixed later in this proof. We choose
\begin{equation*}
 \eta(a)=\left\{\begin{array}{cl}
                (a-a_{0})&\mbox{if}~~a-a_{0}\geq0\,,\\
		0 &\mbox{if}~~a-a_{0}\leq0\
               \end{array}\right.
\end{equation*}
and we aim to deduce
\begin{equation}\label{positive_part_contraction}
 \int_{\theta^{-}(\bar{t})}^{\theta^{+}(\bar{t})}|v(x,\bar{t})-a_{0}|^{+}\,d\sigma\leq C \int_{\Lambda_{t_{1}}^{\bar{t}}}|v_{0}(x)-a_{0}|^{+}\,d\sigma\,,
\end{equation}
from equality (\ref{main_equality_lemma}). 
The non-negativity of $\eta''(a)$ and $q_{\varepsilon}$ implies
\[\int_{\Gamma_{\bar{t}_1}^{\bar{t}}}\int_{\mathbb{R}}\eta''(a)q(x,t,a)\,da\ge 0\quad .\]
Using this inequality in equality (\ref{main_equality_lemma}), we obtain the estimate
\begin{equation*}
\int_{\theta_{\varepsilon}^{-}(\bar{t})}^{\theta^{+}_{\varepsilon}(\bar{t})}\eta_{\varepsilon}(x,\bar{t})\,dx\leq\int_{\Lambda_{\bar{t}_1}^{\bar{t}}}\begin{pmatrix}\eta_{\varepsilon}\\-\xi_{\varepsilon}\end{pmatrix}\cdot\tau\,d\sigma\,. 
\end{equation*}
Letting $\varepsilon\to0^+$ we get
\[\int_{\theta^{-}(\bar{t})}^{\theta^{+}(\bar{t})}\eta(x,\bar{t})\,dx\leq\int_{\Lambda_{t_1}^{\bar{t}}}\begin{pmatrix}\eta(v_1)\\-\xi(v_1)\end{pmatrix}\cdot\tau\,d\sigma\,.\]
We observe
\begin{equation}\label{bound_for_xi}
 \left|\xi(a)\right|\leq \max_{|b|\leq\|v_1\|_{\infty}}|f'(b)|\cdot\eta(a)\,,
\end{equation}
which implies
\[\int_{\theta^{-}(\bar{t})}^{\theta^{+}(\bar{t})}\eta(x,\bar{t})\,dx\leq C\int_{\Lambda_{t_1}^{\bar{t}}}\eta(v_1)\,dx\,.\]
This is our desired result (\ref{positive_part_contraction}) and choosing $a_0=\|v_1\|_{\infty}$ in (\ref{positive_part_contraction}) gives
\[\int_{\theta_{-}(\bar{t})}^{\theta^{+}(\bar{t})}|v(x,\bar{t})-a_{0}|^{+}\,d\sigma=0\]
and thus (\ref{Linfty_bound}) follows:
\[|v(x,t)|\leq \|v_1\|_{\infty}\quad\mbox{a.e. in}\quad\Gamma_{t_1}^{t_2}\,.\]

In order to prove (\ref{control_of_the_defect_measure}), we choose now
\begin{equation*}
 \eta(a):=\left\{\begin{array}{ccl}
                  2a^{2}&\mbox{if}&a\geq-\|v\|_{\infty}\,,\\
		  (a+\|v\|_{\infty})+2\|v\|_{\infty}^{2}&\mbox{if}&-(\|v\|_{\infty}+2\|v\|_{\infty}^{2})\leq a\leq-\|v\|_{\infty}\,,\\
		  0 &\mbox{if}&a\leq-(\|v\|_{\infty}+2\|v\|_{\infty}^{2})\,.
                 \end{array}\right.
\end{equation*}
Since $\eta$ is non-negative, we deduce from (\ref{main_equality_lemma})
\begin{equation}\label{last_equation_for_mass_bound}\begin{split}
 2\int_{\Gamma_{t_{1}+C\varepsilon}^{t_{2}-C\varepsilon}}\int_{\mathbb{R}}q_{\varepsilon}(x,t,a)\,da\,dx\,dt&\leq \int_{\Lambda_{\bar{t}_1}^{\bar{t}}}\begin{pmatrix}\eta_{\varepsilon}\\-\xi_{\varepsilon}\end{pmatrix}\cdot\tau\,d\sigma\quad .
\end{split}\end{equation}
Since $\eta(a)=2a^2$ for $a\in[-\|v\|_{\infty},\|v\|_{\infty}]$ we get
\[ |\xi(a)|=\int_{-\|v\|_{\infty}}^{a}|\eta'(b)f'(b)|\,db
\leq f'(\|v\|_{\infty}) \int_{-\|v\|_{\infty}}^{|a|}|\eta'(b)|\,db\,.\]
Hence, by letting $\varepsilon\to0$ in (\ref{last_equation_for_mass_bound}), we obain
\[\int_{\Gamma_{t_{1}}^{t_{2}}}\int_{\mathbb{R}}q(x,t,a)\,da\,dx\,dt\leq C(\delta+t_2-t_1)\,,\]
as announced in (\ref{control_of_the_defect_measure}).

Finally we are going to prove (\ref{contraction_inequality}). We choose the domain $\Gamma_{t_1}^{t_2}$ in such a way that \[0<\hat{\lambda}\leq \lambda_1\,,\]
where
\begin{equation}\label{def_lambda1}\begin{array}{c}
\lambda_1=\left(\max\left\{f(-R-1-\alpha),f(R+1+\alpha)\right\}\right)^{-1}\\[5mm]
\mbox{and}\\[5mm]
\alpha=\max\{\|w_1\|_{\infty},\|w_2\|_{\infty}\}\,.\end{array}
\end{equation}
 For the two entropy solutions $v_{1},\, v_{2}$ with boundary conditions $w_{1}$ and $w_{2}$ we consider the kinetic equations
\begin{equation*}
\left.\begin{array}{rclcl}
 \partial_{t}\chi_{i}+f'(a)\partial_{x}\chi_{i}&=&\partial_{a}q_{i}&\mbox{in}&\mathcal{D}'(\Gamma_{t_{1}}^{t_{2}}\times\mathbb{R})\\
\chi_{i}&=&\chi(w_{i};a)&\mbox{on}&\Lambda_{t_{1}}^{t_{2}}
\end{array}\right\}
\end{equation*}
where $\chi_i=\chi(v_i(x,t);a)$ for $i=1,\,2\,$. Then, as before, we can regularize our kinetic equations with the kernel defined in (\ref{kernel})
\begin{equation*}
\begin{array}{rclcl}
\partial_{t}\chi^{\varepsilon}_{i}+f'(a)\partial_{x}\chi^{\varepsilon}_{i}&=&\partial_{a}q^{\varepsilon}_{i}(x,t,a)&\mbox{in}&\Gamma_{t_{1}+C\varepsilon}^{t_{2}-C\varepsilon}\\
\end{array}
\end{equation*}
where 
\begin{equation*}
 \chi_{i}^{\varepsilon}(x,t,a)=\int_{\Gamma_{t_{1}}^{t_{2}}}\chi(v_{i}(x,t);a)\varphi_{\varepsilon}(x-y,t-s)\,dx\,dt\quad\mbox{for}\quad i=1,\,2
\end{equation*}
and $C>0$ is again chosen such that for $(x,t)\in \Gamma_{t_{1}+C\varepsilon}^{t_{2}-C\varepsilon}$
\[\varphi_{\varepsilon}(x-y,t-s)=0\quad\mbox{for}\quad (y,s)\in\partial\Gamma_{t_1}^{t_2}\,.\]
Then the function $\left(\chi_{1}^{\varepsilon}-\chi_{2}^{\varepsilon}\right)^{2}$ satisfies for $(x,t)\in\Gamma_{t_{1}+C\varepsilon}^{t_{2}-C\varepsilon}$ and almost every $a\in\mathbb{R}$
\begin{equation}\label{regualrized_equation_contraction}
 \partial_{t}\left(\chi_{1}^{\varepsilon}-\chi_{2}^{\varepsilon}\right)^{2}+f'(a)\partial_{x}\left(\chi_{1}^{\varepsilon}-\chi_{2}^{\varepsilon}\right)^{2}=\chi^{\varepsilon}_{1}\partial_{a}q_{1}^{\varepsilon}+\chi_{2}^{\varepsilon}\partial_{a}q_{2}^{\varepsilon}
-\chi_{2}^{\varepsilon}\partial_{a}q_{1}^{\varepsilon}-\chi_{1}^{\varepsilon}\partial_{a}q_{2}^{\varepsilon}\,.
\end{equation}
We make use again of the following abbreviation: $t_1+C\varepsilon=\bar{t}_1$. Let $\bar{t}\in(\bar{t}_1,+t_2-C\varepsilon)$, then we  integrate (\ref{regualrized_equation_contraction}) in $\Gamma_{\bar{t}_1}^{\bar{t}}\times\mathbb{R}$, which leads to
\begin{multline}\label{result1_for_contraction}
 \int_{\Gamma_{\bar{t}_{1}}^{\bar{t}}}\int_{\mathbb{R}}\partial_{t}\left(\chi_{1}^{\varepsilon}-\chi_{2}^{\varepsilon}\right)^{2}+f'(a)\partial_{x}\left(\chi_{1}^{\varepsilon}-\chi_{2}^{\varepsilon}\right)^{2}\,da\,dx\,dt\\
=\int_{\Gamma_{\bar{t}_{1}}^{\bar{t}}}\int_{\mathbb{R}}\chi^{\varepsilon}_{1}\partial_{a}q_{1}^{\varepsilon}+\chi_{2}^{\varepsilon}\partial_{a}q_{2}^{\varepsilon}-\chi_{2}^{\varepsilon}\partial_{a}q_{1}^{\varepsilon}-\chi_{1}^{\varepsilon}\partial_{a}q_{2}^{\varepsilon}\,da\,dx\,dt\,.
\end{multline}
We recall, that $\chi(v;a)=\mathbbm{1}_{a\leq v}$ and 
\begin{equation*}
q_{1}(x,t,a)=q_{2}(x,t,a)=0\quad\mbox{for}\quad|a|\geq \max\{\|v_{1}\|_{\infty},\|v_{2}\|_{\infty}\}\,.
\end{equation*}
Therefore we can calculate for $(x,t)\in \Gamma_{\bar{t}_1}^{\bar{t}}$ and $i,\,j\in\{1,\,2\}$
\begin{align*}
\int_{\mathbb{R}}\chi_{i}^{\varepsilon}\partial_{a}q_{j}^{\varepsilon}\,da&=\int_{\mathbb{R}}\int_{\Gamma_{t_{1}}^{t_{2}}}\chi(v_{i}(y,s);a)\varphi_{\varepsilon}(x-y,t-s)q_{j}^{\varepsilon}(x,t,a)\,dy\,ds\,da\\
&=\int_{\Gamma_{t_{1}}^{t_{2}}}q_{j}^{\varepsilon}(x,t,v_{i}(y,s))\varphi_{\varepsilon}(x-y,t-s)\,dy\,ds\,.
\end{align*}
This implies, since $\varphi_{\varepsilon}$ and $q_{\varepsilon}$ are non-negative
\[
 \int_{\mathbb{R}}\chi_{2}^{\varepsilon}\partial_{a}q_{1}^{\varepsilon}\,da\geq0\quad\mbox{and}\quad  \int_{\mathbb{R}}\chi_{1}^{\varepsilon}\partial_{a}q_{2}^{\varepsilon}\,da\geq0\,,
\]
which applied in (\ref{result1_for_contraction}) leads to the inequality
\begin{multline}\label{result2_for_contraction}
\int_{\Gamma_{\bar{t}_{1}}^{\bar{t}}}\int_{\mathbb{R}}\partial_{t}\left(\chi_{1}^{\varepsilon}-\chi_{2}^{\varepsilon}\right)^{2}+f'(a)\partial_{x}\left(\chi_{1}^{\varepsilon}-\chi_{2}^{\varepsilon}\right)^{2}\,da\,dx\,dt\\
\leq\int_{\Gamma_{\bar{t}_{1}}^{\bar{t}}}\int_{\mathbb{R}}\chi^{\varepsilon}_{1}\partial_{a}q_{1}^{\varepsilon}+\chi_{2}^{\varepsilon}\partial_{a}q_{2}^{\varepsilon}\,da\,dx\,dt\,.
\end{multline}
For the left hand-side of (\ref{result2_for_contraction}) we compute
\begin{equation}\label{calculation0}\begin{split}
 \int_{\mathbb{R}}\int_{\Gamma_{\bar{t}_{1}}^{\bar{t}}}&\partial_{t}\left(\chi_{1}^{\varepsilon}-\chi_{2}^{\varepsilon}\right)^{2}+f'(a)\partial_{x}\left(\chi_{1}^{\varepsilon}-\chi_{2}^{\varepsilon}\right)^{2}\,dx\,dt\,da\\
&=\int_{\mathbb{R}}\int_{\theta^{-}_{\varepsilon}(\bar{t})}^{\theta_{\varepsilon}^{+}(\bar{t})}\left(\chi_{1}^{\varepsilon}-\chi_{2}^{\varepsilon}\right)^{2}(x,\bar{t})\,dx\,da -\int_{\mathbb{R}}\int_{-\delta}^{\delta}\left(\chi_{1}^{\varepsilon}-\chi_{2}^{\varepsilon}\right)^{2}(x,\bar{t}_{1})\,dx\,da\\
&\quad+\hat{\lambda}\int_{\mathbb{R}}\int_{\bar{t}_{1}}^{\bar{t}}\left(\chi_{1}^{\varepsilon}-\chi_{2}^{\varepsilon}\right)^{2}(\theta^{+}_{\varepsilon}(t),t)-\left(\chi_{1}^{\varepsilon}-\chi_{2}^{\varepsilon}\right)^{2}(\theta_{\varepsilon}^{-}(t),t)\,dt\,da\\
&\quad+\int_{\mathbb{R}}\int_{\bar{t}_{1}}^{\bar{t}}f'(a)\left[\left(\chi_{1}^{\varepsilon}-\chi_{2}^{\varepsilon}\right)^{2}(\theta_{\varepsilon}^{-}(t),t)-\left(\chi_{1}^{\varepsilon}-\chi_{2}^{\varepsilon}\right)^{2}(\theta_{\varepsilon}^{+}(t),t)\right]\,dt\,da\,,
\end{split}\end{equation}
where $\theta_{\varepsilon}^{\pm}$ are defined in (\ref{theta}).
After a change of variable this expression simplifies to
\begin{multline}\label{calculation1}
 \int_{\mathbb{R}}\int_{\Gamma_{\bar{t}_{1}}^{\bar{t}}}\partial_{t}\left(\chi_{1}^{\varepsilon}-\chi_{2}^{\varepsilon}\right)^{2}+f'(a)\partial_{x}\left(\chi_{1}^{\varepsilon}-\chi_{2}^{\varepsilon}\right)^{2}\,dx\,dt\,da\\
=\int_{\mathbb{R}}\int_{\theta^{-}_{\varepsilon}(\bar{t})}^{\theta_{\varepsilon}^{+}(\bar{t})}\left(\chi_{1}^{\varepsilon}-\chi_{2}^{\varepsilon}\right)^{2}(x,\bar{t})\,dx\,da-\int_{\Lambda_{\bar{t}_1}^{\bar{t}}}\int_{\mathbb{R}}\begin{pmatrix}\left(\chi_{1}^{\varepsilon}-\chi_{2}^{\varepsilon}\right)^{2}\\f'(a)\left(\chi_{1}^{\varepsilon}-\chi_{2}^{\varepsilon}\right)^{2}\end{pmatrix}\cdot \tau\,da\,d\sigma\,.
\end{multline}
Using identity (\ref{calculation1}) in (\ref{result2_for_contraction}) gives
\begin{multline}\label{result3_for_contraction}
\int_{\mathbb{R}}\int_{\theta_{\varepsilon}^{-}(\bar{t})}^{\theta^{+}_{\varepsilon}(\bar{t})}\left(\chi_{1}^{\varepsilon}-\chi_{2}^{\varepsilon}\right)^{2}(x,\bar{t})\,dx\,da\\
\leq\int_{\Gamma_{\bar{t}_{1}}^{\bar{t}}}\int_{\mathbb{R}}\chi^{\varepsilon}_{1}\partial_{a}q_{1}^{\varepsilon}+\chi_{2}^{\varepsilon}\partial_{a}q_{2}^{\varepsilon}\,da\,dx\,dt+\int_{\Lambda_{\bar{t}_1}^{\bar{t}}}\int_{\mathbb{R}}\begin{pmatrix}\left(\chi_{1}^{\varepsilon}-\chi_{2}^{\varepsilon}\right)^{2}\\f'(a)\left(\chi_{1}^{\varepsilon}-\chi_{2}^{\varepsilon}\right)^{2}\end{pmatrix}\cdot \tau\,da\,d\sigma\,.
\end{multline}
We claim
\begin{equation}\label{asymptotic_of_q}
\lim_{\varepsilon\to0^+}\int_{\Gamma_{t_{1}}^{\bar{t}}}\int_{\mathbb{R}}\chi^{\varepsilon}_{i}\partial_{a}q_{i}^{\varepsilon}\,da\,dx\,dt=0\quad\mbox{for}\quad i\in\{1,\,2\}\,.
\end{equation}
\textbf{Proof of Claim (\ref{asymptotic_of_q}).} We consider the function $\chi_{i}^{\varepsilon}-\left(\chi^{\varepsilon}_{i}\right)^{2}$ which satisfies satisfies pointwise for $(x,t)\in\Gamma_{t_1+C\varepsilon}^{t_2-C\varepsilon}$ and almost every $a\in\mathbb{R}$
\begin{equation*}
 \partial_{t}\left[\chi_{i}^{\varepsilon}-\left(\chi_{i}^{\varepsilon}\right)^{2}\right]+f'(a)\partial_{x}\left[\chi_{i}^{\varepsilon}-\left(\chi_{i}^{\varepsilon}\right)^{2}\right]=\partial_{a}q_{i}^{\varepsilon}+2\chi_{i}^{\varepsilon}\partial_{a}q_{i}^{\varepsilon}\,.
\end{equation*}
Integrating this in $\Gamma_{\bar{t}_{1}}^{\bar{t}}\times\mathbb{R}$ leads to
\begin{multline}\label{asymptotic0}
\int_{\mathbb{R}}\int_{\Gamma_{\bar{t}_{1}}^{\bar{t}}}\partial_{t}\left[\chi_{i}^{\varepsilon}-\left(\chi_{i}^{\varepsilon}\right)^{2}\right]+f'(a)\partial_{x}\left[\chi_{i}^{\varepsilon}-\left(\chi_{i}^{\varepsilon}\right)^{2}\right]\,dx\,dt\,da\\=\int_{\mathbb{R}}\int_{\Gamma_{\bar{t}_{1}}^{\bar{t}}}2\chi_{i}^{\varepsilon}\partial_{a}q_{i}^{\varepsilon}\,dx\,dt\,da\,,
\end{multline}
where we made use of the fact, that $q_i^{\varepsilon}$ is compactly supported in $a$.
For the left-hand side of (\ref{asymptotic0}) one can compute following step by step (\ref{calculation0}) and (\ref{calculation1})
\begin{multline}\label{asymptotic1}
\int_{\mathbb{R}}\int_{\Gamma_{\bar{t}_{1}}^{\bar{t}}}\partial_{t}\left[\chi_{i}^{\varepsilon}-\left(\chi_{i}^{\varepsilon}\right)^{2}\right]+f'(a)\partial_{x}\left[\chi_{i}^{\varepsilon}-\left(\chi_{i}^{\varepsilon}\right)^{2}\right]\,dx\,dt\,da\\
=\int_{\mathbb{R}}\int_{\theta^{-}_{\varepsilon}(\bar{t})}^{\theta_{\varepsilon}^{+}(\bar{t})}\left[\chi_{i}^{\varepsilon}-\left(\chi_{i}^{\varepsilon}\right)^2\right](x,\bar{t})\,dx\,da-\int_{\Lambda_{\bar{t}_1}^{\bar{t}}}\int_{\mathbb{R}}\begin{pmatrix}\chi_{i}^{\varepsilon}-\left(\chi_{i}^{\varepsilon}\right)^2\\f'(a)\left[\chi_{i}^{\varepsilon}-\left(\chi_{i}^{\varepsilon}\right)^2\right]\end{pmatrix}\cdot \tau\,da\,d\sigma\,.
\end{multline}
For the right-hand side of (\ref{asymptotic1}) we observe
\begin{multline}\label{asymptotic2}
 \lim_{\varepsilon\rightarrow0^+}\int_{\mathbb{R}}\int_{\theta^{-}_{\varepsilon}(\bar{t})}^{\theta_{\varepsilon}^{+}(\bar{t})}\left[\chi_{i}^{\varepsilon}-\left(\chi_{i}^{\varepsilon}\right)^2\right](x,\bar{t})\,dx\,da\\=\int_{\mathbb{R}}\int_{\theta^{-}(\bar{t})}^{\theta^{+}(\bar{t})}\left[\chi_{i}-\left(\chi_{i}\right)^2\right](x,\bar{t})\,dx\,da
\end{multline}
and
\begin{multline}\label{asymptotic3}
 \lim_{\varepsilon\rightarrow0^+}\int_{\Lambda_{\bar{t}_1}^{\bar{t}}}\int_{\mathbb{R}}\begin{pmatrix}\chi_{i}^{\varepsilon}-\left(\chi_{i}^{\varepsilon}\right)^2\\f'(a)\left[\chi_{i}^{\varepsilon}-\left(\chi_{i}^{\varepsilon}\right)^2\right]\end{pmatrix}\cdot \tau\,da\,d\sigma\\=\int_{\Lambda_{t_1}^{\bar{t}}}\int_{\mathbb{R}}\begin{pmatrix}\chi_{i}-\left(\chi_{i}\right)^2\\f'(a)\left[\chi_{i}-\left(\chi_{i}\right)^2\right]\end{pmatrix}\cdot \tau\,da\,d\sigma\,.
\end{multline}
Since
\[\chi_i=\left(\chi_{i}\right)^2\]
the right-hand side of (\ref{asymptotic2}) and (\ref{asymptotic3}) are zero. Thus
\[\lim_{\varepsilon\rightarrow0^+}\int_{\mathbb{R}}\int_{\theta^{-}_{\varepsilon}(\bar{t})}^{\theta_{\varepsilon}^{+}(\bar{t})}\left[\chi_{i}^{\varepsilon}-\left(\chi_{i}^{\varepsilon}\right)^2\right](x,\bar{t})\,dx\,da-\int_{\Lambda_{\bar{t}_1}^{\bar{t}}}\int_{\mathbb{R}}\begin{pmatrix}\chi_{i}^{\varepsilon}-\left(\chi_{i}^{\varepsilon}\right)^2\\f'(a)\left[\chi_{i}^{\varepsilon}-\left(\chi_{i}^{\varepsilon}\right)^2\right]\end{pmatrix}\cdot \tau\,da\,d\sigma=0\,.\]
With (\ref{asymptotic1}) one concludes
\[\lim_{\varepsilon\rightarrow0^+}\int_{\mathbb{R}}\int_{\Gamma_{\bar{t}_{1}}^{\bar{t}}}\partial_{t}\left[\chi_{i}^{\varepsilon}-\left(\chi_{i}^{\varepsilon}\right)^{2}\right]+f'(a)\partial_{x}\left[\chi_{i}^{\varepsilon}-\left(\chi_{i}^{\varepsilon}\right)^{2}\right]\,dx\,dt\,da=0\,.\]
Finally taking limits on both sides of (\ref{asymptotic0}) we get
\[
 \lim_{\varepsilon\to0^+}\int_{\Gamma_{t_{1}}^{\bar{t}}}\int_{\mathbb{R}}\chi^{\varepsilon}_{i}\partial_{a}q_{i}^{\varepsilon}\,da\,dx\,dt\quad\mbox{for}\quad i\in\{1,\,2\}\,,
\]
as announced.

\medskip

Letting $\varepsilon\to0^+$ in (\ref{result3_for_contraction}) and using (\ref{asymptotic_of_q}) leads to
\begin{equation}\label{last1}
 \int_{\mathbb{R}}\int_{\theta^{-}(\bar{t})}^{\theta^{+}(\bar{t})}\left(\chi_{1}-\chi_{2}\right)^{2}(x,\bar{t})\,dx\,da\\
\leq\int_{\Lambda_{t_1}^{\bar{t}}}\int_{\mathbb{R}}\begin{pmatrix}\left(\chi_{1}-\chi_{2}\right)^{2}\\f'(a)\left(\chi_{1}-\chi_{2}\right)^{2}\end{pmatrix}\cdot \tau\,da\,d\sigma\,.
\end{equation}
We compute 
\begin{equation}\label{last2}
 \int_{\mathbb{R}}\left(\chi(v_{1}(x,t);a)-\chi(v_{2}(x,t);a)\right)^{2}\,da=|v_{1}(x,t)-v_{2}(x,t)|\,,
\end{equation}
and 
\begin{equation}\label{last3}
 \int_{\mathbb{R}}\begin{pmatrix}\left(\chi_{1}-\chi_{2}\right)^{2}\\f'(a)\left(\chi_{1}-\chi_{2}\right)^{2}\end{pmatrix}\cdot \tau\,da=\begin{pmatrix}|w_{1}-w_{2}|\\\operatorname{sign}(w_1-w_2)(f(w_1)-f(w_2))\end{pmatrix}\cdot \tau\,.
\end{equation}
Applying (\ref{last2}) and (\ref{last3}) in (\ref{last1}) gives
\begin{equation}\label{last4}
 \int_{\theta^{-}(\bar{t})}^{\theta^{+}(\bar{t})}|v_{1}(x,\bar{t})-v_{2}(x,\bar{t})|\,dx
\leq\int_{\Lambda_{t_1}^{\bar{t}}}\begin{pmatrix}|w_{1}-w_{2}|\\\operatorname{sign}(w_1-w_2)(f(w_1)-f(w_2))\end{pmatrix}\cdot \tau\,d\sigma\,.
\end{equation}
For the right hand side, we compute
\begin{equation*}\begin{split}
 \int_{\Lambda_{t_1}^{\bar{t}}}&\begin{pmatrix}|w_{1}-w_{2}|\\\operatorname{sign}(w_1-w_2)(f(w_1)-f(w_2))\end{pmatrix}\cdot \tau\,d\sigma\\
&=\int_{s_1}^{s_2}|w_1(s,\gamma(s,t_1))-w_2(s,\gamma(s))|\\
&\quad+\partial_s\gamma(s,t_1) \left[\operatorname{sign}(w_1-w_2)(f(w_1)-f(w_2))\right](s,\gamma(s,t_1)\,ds\\
&=\int_{s_1}^{s_2}|w_1(s,\gamma(s,t_1))-w_2(s,\gamma(s))|\cdot(1\pm\partial_s\gamma(s,t_1)f'(\alpha))\\
&\leq C \int_{s_1}^{s_2}|w_1(s,\gamma(s,t_1))-w_2(s,\gamma(s))|\,ds\,,
\end{split}\end{equation*}
for a function $\alpha$. From (\ref{last4}) we obtain
\begin{equation*}
 \int_{\theta^{-}(\bar{t})}^{\theta^{+}(\bar{t})}|v_{1}(x,\bar{t})-v_{2}(x,\bar{t})|\,dx
\leq C\int_{\Lambda_{t_1}^{\bar{t}}}|w_1(s,\gamma(s,t_1))-w_2(s,\gamma(s))|\,d\sigma
\end{equation*}
as claimed.
\end{proof}

%%%%%%%%%%%%%%%%%%%%%%%%%%%%%%%%%%%%%%%%%%%%%%%%%%%%%%%%%%%%%%%%%%%%%%%%%%%%%%%%%%%%%%%%%%%%%%%%%%%%%%%%%%%%%%%%%%%%%%%%%%%%%%%%%%%
%%%%%%%%%%%%%%%%%%%%%%%%%%%%%%%%%%%%%%%%%%%%%%%%%%%%%%%%%%%%%%%%%%%%%%%%%%%%%%%%%%%%%%%%%%%%%%%%%%%%%%%%%%%%%%%%%%%%%%%%%%%%%%%%%%%
%%%%%%%%%%%%%%%%%%%%%%%%%%%%%%%%%%%%%%%%%%%%%%%%%%%%%%%%%%%%%%%%%%%%%%%%%%%%%%%%%%%%%%%%%%%%%%%%%%%%%%%%%%%%%%%%%%%%%%%%%%%%%%%%%%%
\subsection{Blow up at the points of negative density.}\label{section_blow_up}
In this section we aim to prove the following lemma
\begin{lemma}\label{lemma_positiv_density}
 Let $u\in L^{\infty}(\mathbb{R}\times[0,T))$ be a weak solution of (\ref{scl}), which satisfies (\ref{minimality}). Then for $\mathcal{H}^1$ almost every  $(x_{0},t_{0})\in\mathbb{R}\times(0,T)$
\begin{equation*}
\limsup_{r\rightarrow 0^+}\frac{1}{r}\int_{\mathbb{R}}m\left(B_r(x_0,t_0),a\right)\,da\geq 0\,.
\end{equation*}
\end{lemma}
A useful lemma that will be used to prove Lemma \ref{lemma_positiv_density} is the following.
\begin{lemma}\label{compactness}
 Let $u\in L^{\infty}(\mathbb{R}\times(0,T))$ be a weak solution of (\ref{scl}). Let $r_{n}\rightarrow0^+$. For  $(x_{0},t_{0})\in\mathbb{R}\times(0,T)$ define
\begin{equation*}\begin{array}{c}
u_{n}(x,t):=\left(D_n^{-1}\right)^{*}u(x,t)\\[5mm]
\mbox{and}\\[5mm]
\mu_{n}:=\displaystyle\frac{1}{r_n}\int_{\mathbb{R}}\left(D_n\right)_*m\,da\,,
\end{array}\end{equation*}
where
\begin{equation}\label{dilation}
D_n(x,t)=\left(\frac{x-x_0}{r_n},\frac{t-t_0}{r_n}\right)\,.
\end{equation}
Then there exists for $\mathcal{H}^1$ almost every  $(x_{0},t_{0})\in\mathbb{R}\times(0,T)$ a subsequence $r_{k}$ such that
\begin{equation*}
 u_{k}\rightarrow u_{\infty}\quad\mbox{in}\quad L^{1}_{loc}(\mathbb{R}^2)\,.
\end{equation*}
And furthermore
\begin{equation*}
 \mu_{k}\rightharpoondown\mu_{\infty}\quad\mbox{in}\quad\mathcal{M}_{loc}(\mathbb{R}^{2})\,.
\end{equation*}
Which means in other words
\[\int_{\mathbb{R}^2}\psi\,d\mu_k\rightarrow\int_{\mathbb{R}^2}\psi\,d\mu_{\infty}\quad\mbox{for all}\quad\psi\in C^{0}_c(\mathbb{R}^2)\,.\]
\end{lemma}
Lemma \ref{compactness} will be a consequence of of the following proposition, which is proved in Appendix A of \cite{Le}.
\begin{proposition}\label{prop_lecumberry}
 For any constant $M\geq0$, for any bounded set $\Omega$, the set 
\begin{equation*}
 \left\{u\in L^{\infty}(\Omega):~\|u\|_{\infty}+\int_{\Omega}\int_{\mathbb{R}}|m(x,t,a)|\leq M\right\}
\end{equation*}
is compact in $L^{1}(\Omega)$ with respect to the strong topology.
\end{proposition}

\begin{proof}[\textbf{Proof of Lemma \ref{compactness}}]
By construction we already have
\begin{equation}\label{L_infinity_bound_for_u}
 \|u_{n}\|_{\infty}\leq\|u\|_{\infty}\,.
\end{equation}
For this reason it remains to show that for all $R>0$ $|\mu_{n}|(B_{R}(0,0))$ and for $\mathcal{H}^1$ almost every $(x_{0},t_{0})\in\mathbb{R}\times(0,T)$ there exists a constant $C>0$, such that
\begin{equation}\label{bound_measure}
 \limsup_{n\rightarrow\infty}\mu_{n}(B_{r_n}(0,0))\leq C\,.
\end{equation}
But this is a direct consequence of Theorem 2.56 in \cite{AFP}.
Since (\ref{L_infinity_bound_for_u}) and (\ref{bound_measure}) hold, the assumptions of Proposition \ref{prop_lecumberry} are fulfilled and we can extract a subsequence $r_{k'}$ such that
\begin{equation*}
 u_{k'}\rightarrow u_{\infty}\quad\mbox{in}\quad L^{1}_{loc}(\mathbb{R}^2)\,.
\end{equation*}
Additionally we have by the weak$^{*}$ compactness of measures (see Theorem 1.59 in \cite{AFP}), that, possibly after extracting a further subsequence $r_{k}$,
\begin{equation*}
\mu_{k}\rightharpoondown\mu_{\infty}\quad\mbox{in}\quad\mathcal{M}\,,
\end{equation*}
Altogether we have for the sequence $r_{k}$
\begin{equation*}
 u_{k}\rightarrow u_{\infty}\quad\mbox{in}\quad L^{1}_{loc}(\mathbb{R}^2)\,
\end{equation*}
and
\begin{equation*}
 \mu_{k}\rightharpoondown\mu_{\infty}\quad\mbox{in}\quad\mathcal{M}_{loc}(\mathbb{R}^{2})\,,
\end{equation*}
which is what we aimed to prove.
\end{proof}
\begin{proof}[\textbf{Proof of Lemma \ref{lemma_positiv_density}}]
We argue by contradiction. Therefore we assume that there exists a point $(x_{0},t_{0})$ such that 
\begin{equation}\label{contradiction}
\limsup_{r\rightarrow 0^+}\frac{1}{r}\int_{B_{r}((x_{0},t_{0}))}\int_{\mathbb{R}}m(x,t,a)\,da\,dx\,dt<0\,.
\end{equation}
For a sequence $r_{n}\rightarrow0^{+}$ we define
\begin{equation*}\begin{array}{c}
u_{n}(x,t):=\left(D_{n}^{-1}\right)^{*}u(x,t)\\[5mm]
\mbox{and}\\[5mm]
\mu_{n}:=\displaystyle\frac{1}{r_n}\int_{\mathbb{R}}\left(D_n\right)_*m\,da\,.
\end{array}\end{equation*}
Let $u_{k}$ and $\mu_{k}$ be the subsequences given by Lemma \ref{compactness}, with limits $u_{\infty}$, $\mu_{\infty}$. Then we have by strong convergence, that $u_{\infty}$ is a weak solution of
\begin{equation*}
 \partial_{t}u_{\infty}+\partial_{x}f(u_{\infty})=0\,.
\end{equation*}
Furthermore, by the uniqueness of the distributional limit, we conclude that
\begin{equation*}
 \mu_{\infty}=\int_{\mathbb{R}}\partial_{t}(u_{\infty}\wedge a)+\partial_{x}f(u_{\infty}\wedge a)\,da\,.
\end{equation*}
From (\ref{contradiction}) we want to conclude now, that 
\begin{equation}\label{negative_of_limit}
\mu_{\infty}(B_R(0,0))<0\quad\mbox{for all}\quad R>0\,.
\end{equation}

\medskip

\noindent\textbf{Proof of (\ref{negative_of_limit}).} For the sake of contradiction, we assume, that there exists a $R_0$ such that
\[ \mu_{\infty}(B_{R_0}(0,0))\geq0\,.\]
In \cite{Le} it is proved, that there exits a set $K$, which is either a line, or a half-line, or the empty set, such that
\begin{equation}\label{structure_of_measure}
 \partial_{t}u_{\infty}\wedge a+\partial_{x}f(u_{\infty}\wedge a)=\left((X(u_{\infty}^{+}\wedge a)-X(u_{\infty}^{-}\wedge a)\right)\cdot\omega_{K}\,\mathcal{H}^{1}\res K\,,
\end{equation}
where
\begin{equation}\label{vectorfield_X}
 X(u)=\begin{pmatrix}f(u)\\ u\end{pmatrix}\quad\mbox{and}\quad  \omega_{K}=\frac{|u_{\infty}^{+}-u_{\infty}^{-}|}{|X(u^{+}_{\infty})-X(u_{\infty}^{-})|}\begin{pmatrix}1\\ -\frac{f(u_{\infty}^{+})-f(u_{\infty}^{-})}{u_{\infty}^{+}-u_{\infty}^{-}}\end{pmatrix}\,.
\end{equation}
Moreover therein it is proved, that $u_{\infty}$ is $\mathcal{H}^{1}$-a.e. approximately continuous in $K^{c}$ and has $\mathcal{H}^{1}$-a.e. constant approximate jump points $u_{\infty}^{\pm}$ on $K\,$.

A short calculation reveals
\begin{multline}\label{sign_of_measure}
\int_{\mathbb{R}} \left((X(u_{\infty}^{+}\wedge a)-X(u_{\infty}^{-}\wedge a)\right)\cdot\omega_{K}\\=\operatorname{sign}(u^{-}_{\infty}-u^+_{\infty})\int_{\min\{u_{\infty}^{+},u_{\infty}^{-}\}}^{\max\{u_{\infty}^{+},u_{\infty}^{-}\}}\frac{f(u_{\infty}^{+})+f(u_{\infty}^{-})}{2}-f(a)\,da\,.
\end{multline}
By convexity of $f$ we get
\[\frac{f(u_{\infty}^{+})+f(u_{\infty}^{-})}{2}>f(a)\quad\mbox{for}\quad a\in\left[\min\{u_{\infty}^{+},u_{\infty}^{-}\},\max\{u_{\infty}^{+},u_{\infty}^{-}\}\right]\,.\]
This and (\ref{sign_of_measure}) imply, that sign of $\mu_{\infty}$ is completely determined by $\operatorname{sign}(u^{-}_{\infty}-u^{+}_{\infty})$. 
Henceforth
\[ \mu_{\infty}(B_{R_0}(0,0))\geq0\]
can only be fulfilled, if
\[u^-_{\infty}\geq u^+_{\infty}\,.\]
But this implies that the measure is $\mu_{\infty}$ has a sign, i.e.
\[\mu_{\infty}\geq0\,.\]
Let $\mu_{k}^{\pm}$ be the positive respective negative part of $\mu_{k}$, i.e. $\mu_{k}^{\pm}$ are non-negative measures such that
\[\mu_{k}=\mu_k^+-\mu_k^-\,.\] Then we can extract a further subsequence $k'$ such that 
\[\mu_{k'}^{+}\rightharpoondown\nu^{+}\quad\mbox{and}\quad\mu_{k'}^{-}\rightharpoondown\nu^{-}\quad\mbox{in}\quad\mathcal{M}_{loc}(\mathbb{R}^2)\,.\]
For $R>0$ and  non-negative $\psi\in C^{\infty}_{c}(B_{R}(0,0))$ we get
\[\int_{B_{R_0}(0,0)}\psi\,d\mu_{\infty}=\lim_{k'\rightarrow\infty}\int_{B_{R}(0,0)}\psi\,d\mu_{k'}=\int_{B_{R}(0,0)}\psi\,d\nu^+-\int_{B_{R}(0,0)}\psi\,d\nu^-\,.\]
Since $\mu_{\infty}$ is non-negative we get for all non-negative $\psi\in C^{\infty}_{c}(B_{R_0}(0,0))$
\[\int_{B_{R_0}(0,0)}\psi\,d\nu^-\leq\int_{B_{R}(0,0)}\psi\,d\nu^+\,.\]
Hence
\begin{equation}\label{inequality_for_negative_and_positive_part}
 \nu^{-}(B_{R}(0,0))\leq \nu^{+}(B_{R}(0,0))
\end{equation}
By Theorem 1.2 in \cite{Le} (see also Theorem 1.1 in \cite{AKLR}) we have for a rectifiable set $J_u$ and an ${\mathcal H}^1$ measurable function $h:~J_u\to\mathbb{R}$
\begin{equation}\label{rectifiable}
\int_{\mathbb{R}}|m(x,t,a)|\,da=h\cdot\mathcal{H}^1\res J_u+\delta_u\,,
\end{equation}
where $\delta_u$ satisfies
\[\forall~B\quad\mbox{Borel}\quad \mathcal{H}^1(B)<\infty\Longrightarrow\delta_u(B)=0\,.\]
Therefore we can choose $R_1$, such that for all $k'$
\[\mu^{-}_{k'}(\partial B_{R_1}(0,0))\leq\frac{1}{r_{k'}}\int_{D_{k'}^{-1}(\partial B_{R_{1}}(0,0))}h\,d\mathcal{H}^1\res J_u=0\,.\]
Hence
\[\nu^{-}(\partial B_{R_1}(0,0))= \lim_{k'\rightarrow\infty}\mu^{-}_{k'}(\partial B_{R_1}(0,0))=0\,.\]
This and (\ref{inequality_for_negative_and_positive_part}) imply
\begin{align*}
 \limsup_{k'\rightarrow\infty}\mu_{k'}^-(B_{R_1}(0,0)&\leq\nu^-(\bar{B}_{R_1}(0,0))=\nu^-(B_{R_1}(0,0))\\
&\leq\nu^{+}(B_{R_1}(0,0))\leq\liminf_{k'\rightarrow\infty}\mu_{k'}^{+}(B_{R_1}(0,0))\,.
\end{align*}
\begin{align*}
\limsup_{k'\rightarrow\infty}\mu_{k'}(B_{R_1}(0,0))\geq\liminf_{k'\rightarrow\infty}\mu_{k'}^+B_{R_1}(0,0))-\limsup_{k'\rightarrow\infty}\mu^-_{k'}(B_{R_1}(0,0))\geq0\,,
\end{align*}
which obviously contradicts (\ref{contradiction}) and we get (\ref{negative_of_limit}).

Inequality (\ref{negative_of_limit}) implies, that the set $K$ in (\ref{structure_of_measure}) is non-empty and \[\mu_{\infty}<0\,,\]
which gives, again from above considerations
\[u_{\infty}^-<u_{\infty}^+\,.\]
Moreover the convexity of $f$ implies for every $a\in(u_{\infty}^{-},u_{\infty}^{+})$
\begin{multline*}
 \partial_{t}u_{\infty}\wedge a+\partial_{x}f(u_{\infty}\wedge a)\\=\left(\frac{f(a)-f(u_{\infty}^{-})}{a-u_{\infty}^{-}}-\frac{f(u_{\infty}^{+})-f(u_{\infty}^{-})}{u^+_{\infty}-u^-_{\infty}}\right)\left(a-u_{\infty}^{-}\right)\mathcal{H}^{1}\res K\leq0\,.
\end{multline*}
In other words, we get
\begin{equation*}
  \partial_{t}u_{\infty}\wedge a+\partial_{x}f(u_{\infty}\wedge a)\leq0\,.
\end{equation*}
For $P=(x_{p},t_{p})\in\mathbb{R}^{2}$ let $K=P+\mathbb{R}\omega_{K}^{\perp}$ if $K$ is a line or $K=P+\mathbb{R}_{+}\omega_{K}^{\perp}$ if $K$ is a halfline. Define
\begin{equation*}\begin{array}{c}
H^{+}:=\left\{(x,t):~((x,t)-P)\cdot\omega_{K}>0\right\}\\[5mm]
\mbox{and}\\[5mm]
H^{-}:=\left\{(x,t):~((x,t)-P)\cdot\omega_{K}<0\right\}
\end{array}\end{equation*}
if $K$ is a line and
\begin{equation*}\begin{array}{c}
 H^{+}:=\left\{(x,t):~((x,t)-P)\cdot\omega_{K}>0~\mbox{and}~x>f'(u_{\infty}^{+})(t-t_{p})+x_{p}\right\}\\[5mm]
H^{-}:=\left\{(x,t):~((x,t)-P)\cdot\omega_{K}<0~\mbox{and}~x<f'(u_{\infty}^{-})(t-t_{p})+x_{p}\right\}\,,
\end{array}
\end{equation*}
if $K$ is a half-line. From the proof of Proposition 3.3 in \cite{Le} (see also Theorem 6.2 in \cite{AKLR} for a similar proof) we get, that
\begin{equation*}
 u_{\infty}(x,t)=u_{\infty}^{-}\quad\mbox{on}\quad H^{-}\quad\mbox{and}\quad u_{\infty}(x,t)=u_{\infty}^{+}\quad\mbox{on}\quad H^{+}\,.
\end{equation*}
Now we choose $\bar{t}\in\mathbb{R}$ and $\delta>0$ in the definition of the sets $\Lambda_{\bar{t}}^{\bar{t}+1}$ and $\Gamma_{\bar{t}}^{\bar{t}+1}$ (see (\ref{the_set_gamma})), in such a way that 
\begin{equation*}
\left[-\frac{\delta}{2},\frac{\delta}{2}\right]\times\{t\}\cap K\neq\emptyset\quad\forall~t\in(\bar{t},\bar{t}+1)\,.
\end{equation*}
Furthermore $\Gamma_{\bar{t}}^{\bar{t}+1}$ is defined such that the conclusions of Lemma \ref{properties_of_entropy_solutions} applies to this trapeze.
In particular the strong convergence of $u_{k}$ in $L^{1}_{loc}(\mathbb{R}^{2})$ implies
\begin{equation*}
 u_{k}\rightarrow u_{\infty}\quad\text{in}\quad L^{1}\left(\Gamma_{\bar{t}}^{\bar{t}+1}\right)\,,
\end{equation*}
which directly implies by a change of variable
\begin{align*}
\int_{\bar{t}}^{\bar{t}+1}\int_{\Lambda_{t'}^{\bar{t}+1}}|u_{k}-u_{\infty}|\,d\sigma\,dt' \rightarrow 0\,.
\end{align*}
Thus for almost every $t_{1}\in(\bar{t},\bar{t}+1)$ we get
\begin{equation}\label{convergence_of_inital_data}
 \int_{\Lambda_{t_{1}}^{\bar{t}+1}}|u_{k}-u_{\infty}|\,d\sigma\rightarrow0\,
\end{equation}
and moreover by (\ref{rectifiable})
\begin{equation}\label{no_measure}
\mu_{k}(\Lambda_{t_1}^{\bar{t}+1})=\int_{D_k^{-1}(\Lambda_{t_1}^{\bar{t}+1})}h\,\mathcal{H}^1\res J_{u}=0\,.
\end{equation}
We set $t_{2}:=\bar{t}+1$, then, according to Proposition \ref{existence}, we can choose a $t_1\in(\bar{t},\bar{t}+1)$ such that for all $k\in\mathbb{N}$ (\ref{convergence_of_inital_data}), (\ref{no_measure}) holds and for $k\in\mathbb{N}\cup\{\infty\}$ there exists an entropy solution $w_{k}$ of
\begin{equation}\label{problem_for_uk}
\left.\begin{array}{rclcl}
\partial_{t}w_{k}+\partial_{x}f(w_{k})&=&0&\mbox{in}&\Gamma_{t_{1}}^{t_{2}}\,,\\[2mm]
w_{k}&=&u_{k}&\mbox{on}&\Lambda_{t_{1}}^{t_{2}}\,.
\end{array}\right\}
\end{equation}
By Lemma \ref{properties_of_entropy_solutions} we have for all $t_{1}\leq t<t_{2}$
\begin{equation*}
\int_{\theta^{-}(t)}^{\theta^{+}(t)}|w_{k}(x,t)-w_{\infty}(x,t)|\,dx\leq\int_{\Lambda_{t_{1}}^{t_{2}}}|u_{k}-u_{\infty}|\,d\sigma\,.
\end{equation*}
This and (\ref{convergence_of_inital_data}) imply
\begin{equation*}
 w_{k}\rightarrow w_{\infty}\quad\mbox{in}\quad L^{1}\left(\Gamma_{t_{1}}^{t_{2}}\right)\,.
\end{equation*}
By our choice of $t_{1}$, we have for an $x_{1}\in\left[-\frac{\delta}{2},\frac{\delta}{2}\right]$
\begin{equation*}
 u_{\infty}(x,t_{1})=\left\{\begin{array}{rl}
			 u_{\infty}^{-} &\mbox{if}~x<x_{1}\\
			u_{\infty}^{+} &\mbox{if}~x>x_{1}\,.\\
\end{array}\right.
\end{equation*}
This structure of $u_{\infty}$ at the time $t_1$ allows us to compute $w_{\infty}$ explicitly. Since $u_{\infty}^{-}<u_{\infty}^{+}$, the two states $u_{\infty}^{-}$ and $u_{\infty}^{+}$ are connected  by a rarefaction wave 
\begin{equation*}
 w_{\infty}(x,t):=\left\{\begin{array}{lcl}
                          u_{\infty}^{-}&\mbox{if}&x-x_{1}<f'(u_{\infty}^{-})(t-t_{1})\,,\\
			  \left(f'\right)^{-1}\left(\frac{x-x_{1}}{t-t_{1}}\right)&\mbox{if}&f'(u_{\infty}^{-})(t-t_{1})<x-x_{1}<f'(u_{\infty}^{+})(t-t_{1})\,,\\
			 u_{\infty}^{+}&\mbox{if}&x-x_{1}>f'(u_{\infty})^{+}(t-t_{1})\,.
                         \end{array}\right.
\end{equation*}
We observe, that $w_{\infty}$ is a Lipschitz function and this implies pointwise almost everywhere in $\Gamma_{t_{1}}^{t_{2}}$ 
\begin{equation*}
 \partial_{t}w_{\infty}+\partial_{x}f(w_{\infty})=0\,.
\end{equation*}
Hence
\begin{align*}
 q_{\infty}(x,t,a)&=\partial_{t}(w_{\infty}\wedge a)+\partial_{x}f(w_{\infty}\wedge a)\\
&=\mathbbm{1}_{w\leq a}\left[\partial_{t}w_{\infty}+f'(w_{\infty}\wedge a)\partial_{x}w_{\infty}\right]=0\quad\mbox{in}\quad\Gamma_{t_1}^{t_2}\,.
\end{align*}
Furthermore the strong convergence of $w_{k}$ in $L^{1}(\Gamma_{t_{1}}^{t_{2}})$ implies
\begin{equation*}
 q_{k}\rightharpoondown q_{\infty}\quad\mbox{in}\quad\mathcal{M}_{loc}(\mathbb{R}^2)\,,
\end{equation*}
where \[q_k=\partial_t w_k\wedge a+\partial_x f(w_k\wedge a)\,.\]
To simplify notations, we define
\begin{equation*}
 \begin{array}{c}
 \varGamma_{k}:=\left\{(x,t)\in\mathbb{R}\times(0,T):~D_k(x,t)\in\Gamma_{t_1}^{t_2}\right\}\\[5mm]
 \mbox{and}\\[5mm]
 \varLambda_k:=\left\{(x,t)\in\mathbb{R}\times(0,T):~D_k(x,t)\in\Lambda_{t_1}^{t_2}\right\}
 \end{array}\,,
\end{equation*}
where the map $D_k$ is defined in (\ref{dilation}).
Then we define the rescaled function
\begin{equation*}
 \tilde{w}_{k}(x,t)=\left\{
\begin{array}{lcl}
\left(D_k\right)^{*}w_{k}&\mbox{if}&(x,t)\in\varGamma_{k}\,,\\
u&\mbox{if}&(x,t)\in\mathbb{R}\times(0,t_0+r_{k}t_2)\backslash\varGamma_{k}\,.
\end{array}\right.
\end{equation*}
and claim, that $w_{k}\in L^{\infty}(\mathbb{R}\times(0,t_{0}+r_{k}t_{2}))$ is a weak solution of (\ref{scl}) for all $k\in\mathbb{N}$. To do so, we first observe that $u_{k}$ itself is a weak solution of (\ref{problem_for_uk}). With that knowledge we calculate.
\begin{equation*}\begin{split}
\int_{\varGamma_{k}}\tilde{w}_{k}\partial_{t}\psi+f(\tilde{w}_k)\partial_{x}\psi\,dx\,dt&=r^{2}\int_{\Gamma_{t_1}^{t_2}}\tilde{w}_{k}\partial_{t}\psi+f(\tilde{w}_{k})\partial_{x}\psi\,dx\,dt\\
&=-r_k^2\int_{\Lambda_{t_1}^{t_2}}\psi\begin{pmatrix} u_{k} \\
 -f(u_{k})\end{pmatrix}\cdot\tau\,d\sigma\\
&=r_k^{2}\int_{\Gamma_{t_1}^{t_2}}u_{k}\partial_{t}\psi+f(u_{k})\partial_{x}\psi\,dx\,dt\\
&=\int_{\varGamma_{k}}u\partial_{t}\psi+f(u)\partial_{x}\psi\,dx\,dt\,.
\end{split}\end{equation*}
Using this equality we see
\begin{align*}
\int_{\mathbb{R}\times[0,t_{0}+r_{n}t_2]}\tilde{w}_{n}\partial_{t}\psi+f(\tilde{w}_{k})\partial_{x}\psi\,dx\,dt&= \int_{\Gamma_{k}}\tilde{w}_{k}\partial_{t}\psi+f(\tilde{w}_{k})\partial_{x}\psi\,dx\,dt\\
&\quad+\int_{\mathbb{R}\times(0,t_0+r_{k}t_2)\backslash\varGamma_{k}}u\partial_{t}\psi+f(u)\partial_{x}\psi\,dx\,dt\\
&=\int_{\varGamma_{k}}u\partial_{t}\psi+f(u)\partial_{x}\psi\,dx\,dt\\
&\quad+\int_{\mathbb{R}\times(0,t_0+r_{k}t_2)\backslash\varGamma_{k}}u\partial_{t}\psi+f(u)\partial_{x}\psi\,dx\,dt\\
&=\int_{\mathbb{R}\times[0,t_{0}+r_{n}t_2]}u\partial_{t}\psi+f(u)\partial_{x}\psi\,dx\,dt\\
&=\int_{\mathbb{R}}u_{0}(x)\psi(x,0)\,dx\,,
\end{align*}
which means, that $\tilde{w}_k$ is indeed a weak solution of (\ref{scl}).
Therefore the minimality condition (\ref{minimality}) of $u$ applies and we deduce
\begin{equation*}
 \int_{\mathbb{R}\times (0,t_{0}+r_{k}t_{2})\times\mathbb{R}}|m(x,t,a)|\,da\,dx\,dt\leq \int_{\mathbb{R}\times(0,t_{0}+r_{k}t_{2})\times\mathbb{R}}|\tilde{q}_{k}(x,t,a)|\,da\,dx\,dt\,.
\end{equation*}
But since \[ m(x,t,a)=\tilde{q}_k(x,t,a)\quad\mbox{on}\quad\mathbb{R}\times(0,t_0+r_{k}t_2)\backslash\bar{\varGamma}_{k}\]
we get
\begin{equation}\label{comparison_measure}
\int_{\mathbb{R}}|m|(\varGamma_{k}\cup\varLambda_k,a)\,da\leq\int_{\mathbb{R}}|\tilde{q}_{k}|(\varGamma_{k}\cup\varLambda_k,a)\,da\,.
\end{equation}
We claim now
\begin{equation}\label{no_measure_on_boundary}
|\tilde{q}_{k}|(\varLambda_k,a)=0\quad\mbox{for all}\quad k\in\mathbb{N}\,.
\end{equation}
\textbf{Proof of (\ref{no_measure_on_boundary}).} We define the domain $\Lambda_{\varepsilon}$ such that
\begin{equation}
\partial\Lambda_{\varepsilon}=\Lambda_{t_1+\varepsilon}^{t_2}\cup\Lambda_{t_1-\varepsilon}^{t_2}\cup I_l\cup I_r
\end{equation}
and
\[\Lambda_{t_1}^{t_2}\subset\Lambda_{\varepsilon}\,,\]
where
\[I_l=\left[\frac{t_2-(t_1+\varepsilon)}{\hat{\lambda}}+\delta,\frac{t_2-(t_1-\varepsilon)}{\hat{\lambda}}\right]\]
and
\[I_r=\left[-\frac{t_2-(t_1-\varepsilon)}{\hat{\lambda}}-\delta,-\frac{t_2-(t_1+\varepsilon)}{\hat{\lambda}}-\delta\right]\]
Then for $\varGamma^{\varepsilon}_{k}:=D_{k}^{-1}(\Lambda_{\varepsilon})$ and $\psi\in C_{c}^{\infty}(\mathbb{R}\times(0,t_0+t_2r_k))$ it follows by Theorem 1.3.4 in \cite{Da}
\begin{equation}\label{boundary1}
\int_{\varLambda^{\varepsilon}_k}\tilde{w}\wedge a\ \partial_t\psi+f(u\wedge a)\partial_x\psi\,dx\,dt=\int_{\partial\varLambda^{\varepsilon}_k}\begin{pmatrix}f(\tilde{w})\\\tilde{w}\end{pmatrix}\cdot n\psi\,d\sigma+\int_{\varLambda^{\varepsilon}_k}\psi\,d\tilde{q}(x,t,a)\,,
\end{equation}
where $n$ is the outer unit normal of $\varLambda_{k}^{\varepsilon}$.
The boundary term can be separated in three parts
\begin{equation}\label{bound2}\begin{split}
\int_{\partial\varLambda^{\varepsilon}_k}\begin{pmatrix}f(\tilde{w})\\\tilde{w}\end{pmatrix}\cdot n\psi\,d\sigma&=\int_{D_k ^{-1}(\Lambda_{t_1-\varepsilon}^{t_2})}\begin{pmatrix}f(u)\\\ u\end{pmatrix}\cdot n\psi\,d\sigma-\int_{D_k ^{-1}(\Lambda_{t_1+\varepsilon}^{t_2})}\begin{pmatrix}f(\tilde{w})\\\tilde{w}\end{pmatrix}\cdot n\psi\,d\sigma\\
&\quad+\int_{D_k^{-1}(I_l)}\tilde{w}(x,t_0+r_kt_2)\,dx+\int_{D_k^{-1}(I_r)}\tilde{w}(x,t_0+r_kt_2)\,dx
\end{split}\end{equation}
As $\varepsilon\to0^+$ the two last quantities in the right-hand side of in (\ref{bound2}) vanish. For the first expression on the right hand side of (\ref{bound2}) one concludes
\begin{equation}\label{bound3}
\lim_{\varepsilon\to0^+}\int_{D_k ^{-1}(\Lambda_{t_1-\varepsilon}^{t_2})}\begin{pmatrix}f(u)\\\ u\end{pmatrix}\cdot n\psi\,d\sigma=\int_{D_k ^{-1}(\Lambda_{t_1}^{t_2})}\begin{pmatrix}f(u)\\\ u\end{pmatrix}\cdot n\psi\,d\sigma\,.
\end{equation}
With a change of variable and with Lemma \ref{properties_of_entropy_solutions} it follows
\begin{equation}\begin{split}\label{bound4}
\lim_{\varepsilon\to0^+}\int_{D_k ^{-1}(\Lambda_{t_1+\varepsilon}^{t_2})}\begin{pmatrix}f(\tilde{w})\\\tilde{w}\end{pmatrix}\cdot n\psi\,d\sigma&=\lim_{\varepsilon\to 0^+ }r_k\int_{\Lambda_{t_1+\varepsilon}^{t_2}}\begin{pmatrix}f(w)\\w\end{pmatrix}\cdot n\psi\,d\sigma\\
&=r_k\int_{\Lambda_{t_1+\varepsilon}^{t_2}}\begin{pmatrix}f(u_k)\\u_k\end{pmatrix}\cdot n\psi\,d\sigma\\
&=\int_{D_k ^{-1}(\Lambda_{t_1}^{t_2})}\begin{pmatrix}f(u)\\\ u\end{pmatrix}\cdot n\psi\,d\sigma\,.
\end{split}\end{equation}
From (\ref{bound2}), (\ref{bound3}) and (\ref{bound4}) we conclude
\begin{equation*}
\lim_{\varepsilon\to0^+}\int_{\partial\varLambda^{\varepsilon}_k}\begin{pmatrix}f(\tilde{w})\\\tilde{w}\end{pmatrix}\cdot n\psi\,d\sigma=0\,.
\end{equation*}
Therefore we can conclude from (\ref{boundary1})
\[\lim_{\varepsilon\to0^+}\int_{\Lambda^{\varepsilon}_{k}}\psi\ d\tilde{q}(x,t,a)=0\quad\mbox{for}\quad \psi\in C_c^{\infty}(\mathbb{R}\times(0,t_0+r_kt_2))\,.\]
From this it follows,
\[|\tilde{q}_k|(\varLambda_k,a)=0\]
as claimed.

\medskip

In a next step we show, that (\ref{no_measure_on_boundary}) induces 
\begin{equation}\label{limit}
\lim_{k\rightarrow\infty}\int_{\mathbb{R}}q_{k}(\Gamma_{t_1}^{t_2},a)\,da=\int_{\mathbb{R}}q_{\infty}(\Gamma_{t_1}^{t_2},a)\,da\,.
\end{equation}
\textbf{Proof of (\ref{limit}).} Since $w_k$ is an entropy solution we deduce from (\ref{no_measure_on_boundary}) that $|\tilde{q}_k|(\partial \varGamma_k,a)=0$ and therefore
\begin{equation}\label{limit2}
\frac{1}{r_k}\int_{\mathbb{R}}\left(D_k\right)_{*}|\tilde{q}_k|(\partial\Gamma_{t_1}^{t_2},a)\,da=0\,.
\end{equation}
Lemma \ref{properties_of_entropy_solutions} and (\ref{limit2}) imply for a constant $C>0$
\begin{align*}
\frac{1}{r_k}\int_{\mathbb{R}}\left(D_k\right)_{*}|\tilde{q}_k|(\bar{\Gamma}_{t_1}^{t_2},a)\,da&=\frac{1}{r_k}\int_{\mathbb{R}}\left(D_k\right)_{*}|\tilde{q}_k|(\partial\Gamma_{t_1}^{t_2},a)\,da+\int_{\mathbb{R}}q_k(\Gamma_{t_1}^{t_2},a)\,da\\
&=\int_{\mathbb{R}}q_k(\Gamma_{t_1}^{t_2},a)\,da<C\,.
\end{align*}
Hence one gets for a positive measure $\nu\in\mathcal{M}(\bar{\Gamma}_{t_1}^{t_1})$ after possibly extracting a subsequence
\begin{equation*}
\frac{1}{r_k}\int_{\mathbb{R}}\left(D_k\right)_{*}|\tilde{q}_k|\rightharpoondown\nu\quad\mbox{in}\quad\mathcal{M}(\bar{\Gamma}_{t_1}^{t_1})\,.
\end{equation*}
Then Proposition 1.62 in \cite{AFP} and (\ref{limit2}) imply
\begin{equation}\label{limit3}
\lim_{k\rightarrow\infty}\frac{1}{r_k}\int_{\mathbb{R}}\left(D_k\right)_{*}|\tilde{q}_k|(\partial\Gamma_{t_1}^{t_2},a)\,da=\nu(\partial\Gamma_{t_1}^{t_2})=0\,.
\end{equation}
But $\nu(\partial\Gamma_{t_1}^{t_2})=0$  and Proposition 1.62 in \cite{AFP} give again
\begin{equation*}
\lim_{k\rightarrow\infty}\frac{1}{r_k}\int_{\mathbb{R}}\left(D_k\right)_{*}|\tilde{q}_k|(\Gamma_{t_1}^{t_2},a)\,da=\lim_{k\rightarrow\infty}\int_{\mathbb{R}}q_{k}(\Gamma_{t_1}^{t_2},a)\,da=\int_{\mathbb{R}}q_{\infty}(\Gamma_{t_1}^{t_2},a)\,da\,.
\end{equation*}

\medskip

Since (\ref{no_measure}) and (\ref{no_measure_on_boundary}) holds we deduce from (\ref{comparison_measure})
\begin{equation*}
|\mu_k|(\Gamma_{t_1}^{t_2})\leq\int_{\mathbb{R}}q_k(\Gamma_{t_1}^{t_1},a)\,da\,.
\end{equation*}
Taking the limit on both sides and applying (\ref{limit}) gives
\begin{align*}
|\mu_{\infty}|(\Gamma_{t_1}^{t_2})&\leq \liminf_{k\rightarrow+\infty}|\mu_{k}|(\Gamma_{t_1}^{t_2})\leq \liminf_{k\rightarrow+\infty}\int_{\mathbb{R}}q_{k}(\Gamma_{t_1}^{t_2},a)\,da\\
&=\int_{\mathbb{R}}q_{\infty}(\Gamma_{t_1}^{t_2},a)\,da=0\,.
\end{align*}
But \[ |\mu_{\infty}|(\Gamma_{t_1}^{t_2})=0\] is contradiction to (\ref{negative_of_limit}). Therefore 
\[\limsup_{r\rightarrow 0^+}\frac{1}{r}\int_{\mathbb{R}}m(B_{r}((x_{0},t_{0})),a)\,da\geq 0\,,\]
which is, what we aimed to prove.
\end{proof}
%%%%%%%%%%%%%%%%%%%%%%%%%%%%%%%%%%%%%%%%%%%%%%%%%%%%%%%%%%%%%%%%%%%%%%%%%%%%%%%%%%%%%%%%%%%%%%%%%%%%%%%%%%%%%%%%%%%%%%%%%%%%%%%%%%%
%%%%%%%%%%%%%%%%%%%%%%%%%%%%%%%%%%%%%%%%%%%%%%%%%%%%%%%%%%%%%%%%%%%%%%%%%%%%%%%%%%%%%%%%%%%%%%%%%%%%%%%%%%%%%%%%%%%%%%%%%%%%%%%%%%%
%%%%%%%%%%%%%%%%%%%%%%%%%%%%%%%%%%%%%%%%%%%%%%%%%%%%%%%%%%%%%%%%%%%%%%%%%%%%%%%%%%%%%%%%%%%%%%%%%%%%%%%%%%%%%%%%%%%%%%%%%%%%%%%%%%%
\subsection{Proving that $u$ is entropic}
In this last section we are going to prove
\begin{lemma}\label{proving_entropy}
 Let $u\in L^{\infty}(\mathbb{R}\times [0,T)$ be a weak solution of (\ref{scl}). Let $m(x,t,a)$ its entropy defect measure. If for $\mathcal{H}^1$ almost every $(x_{0},t_{0})\in\mathbb{R}\times(0,T)$ 
\begin{equation}\label{positivity_assumption}
 \limsup_{r\rightarrow0^+}\frac{1}{r}\int_{B_{r}(x_{0},t_{0})}\int_{\mathbb{R}}m(x,t,a)\,da\,dx\,dt\geq0\,,
\end{equation}
then $u$ is the entropy solution of (\ref{scl}).
\end{lemma}
\begin{proof}[\textbf{Proof of Lemma \ref{proving_entropy}}.] We follow closely \cite{ALR}. Without loss of generality we can assume $f(0)=0$ and $f\geq0$. According to Theorem \ref{correspondence} there exists a $g\in W^{1,\infty}(\mathbb{R}\times[0,T))$ such that $u=\partial_{x}g$ and it satisfies almost everywhere
\begin{equation}\label{HJ}
\left.
\begin{array}{rl}
 \partial_{t}g+f(\partial_{x}g)&=0\,,\\
\partial_{x}g(x,0)&=u_{0}(x)\,.
\end{array}\right\}
\end{equation}
We want to show, that $g$ is a viscosity solution of (\ref{HJ}), i.e. we want to prove, that $g$ is a sub- and supersolution of (\ref{HJ}). This immediately implies by Corollary 1.7.2 in \cite{ALR}, that $u$ is an entropy solution. We already now, that $g$ satisfies (\ref{HJ}) almost everywhere, then Proposition 5.1 in \cite{BC} implies, that $g$ is a subsolution. Therefore it remains to show, that $g$ is a supersolution of (\ref{HJ}).
Let $\psi\in C^{1}(\mathbb{R}\times\mathbb{R}_{+})$ such that $g-\psi$ has a local minimum in $(x_{0},t_{0})$. Without loss of generality we can assume $g(x_{0},t_{0})=\psi(x_{0},t_{0})$. We want to show that
\[ \partial_{t}\psi(x_{0},t_{0})+f\left(\partial_{x}\psi(x_{0},t_{0})\right)\geq0\,.\]
We argue by contradiction, therefore we assume
\[ \partial_{t}\psi(x_{0},t_{0})+f\left(\partial_{x}\psi(x_{0},t_{0})\right)<0\,.\]
Since $f\geq0$ this immediately implies
\begin{equation}\label{psi_is_negative}
\partial_t\psi(x_0,t_0)<0\,.
\end{equation}
For a sequence $r_{n}\rightarrow0^{+}$ we introduce
\begin{align*}
 u_{n}(x,t)&=u(x_{0}+r_{n}x,t_{0}+r_{n}t),\\
 \psi_{n}(x,t)&=\frac{1}{r_{n}}\left(\psi(x_{0}+\lambda r_{n}x,t_{0}+r_{n}t)-\psi(x_{0},t_{0})\right),\\
 g_{n}(x,t)&=\frac{1}{r_{n}}\left(g(x_{0}+r_{n}x,t_{0}+r_{n}t)-g(x_{0},t_{0})\right),
\end{align*}
 where $0<\lambda<1$ is a constant, which we choose later. According to Lemma \ref{compactness} we can extract a subsequence $r_{k}$ such that 
\[u_{k}\rightarrow u_{\infty}\quad\mbox{in}\quad L^{1}(B_{1}) \]
Since $\partial_{x}g_{k}=u_{k}$ and $\partial_{t}g_{k}=f(u_{k})$ we have by Arzela-Ascoli, that $g_{k}$ converges uniformly to a Lipschitz function $g_{\infty}$ such that $\partial_{x}u_{\infty}=g_{\infty}$ and $g_{\infty}$ fulfills (\ref{HJ}) almost everywhere. Furthermore we have for $\psi_{\infty}:=\nabla\psi(x_{0},t_{0})\cdot(\lambda x, t)^{\mathrm{T}}$
\begin{equation*}
 \lim_{k\rightarrow\infty}\psi_{k}(x,t)=\psi_{\infty}\,.
\end{equation*}
We notice, that for all $0<\lambda<1$ and for all $k$ the functions $g_{k}-\psi_{k}$ have a local minimum in $(0,0)\,.$ By uniform convergence the function $g_{\infty}-\psi_{\infty}$ admits also a local minimum in $(0,0)\,.$ Moreover 
\begin{equation*}
\mu_{k}=\frac{1}{r_k}\int_{\mathbb{R}}\left(D_k\right)_{*}m\,da\rightharpoondown \mu_{\infty}\quad\mbox{in}\quad\mathcal{M}(B_{1})\,.
\end{equation*}
Similar as in Section \ref{section_blow_up} from
\begin{equation*}
 \lim_{k\rightarrow\infty}\int_{B_{1}(0,0)}\mu_{k}(B_1(0,0))\geq0\,,
\end{equation*}
we can conclude
\begin{equation*}
m_{\infty}(x,t,a):=\partial_{t}u_{\infty}\wedge a+\partial_{x}f(u_{\infty}\wedge a)\geq0\,.
\end{equation*}
Let $\delta>0$, then the function
\begin{equation*}
 h_{\delta}(x,t):=g_{\infty}-\psi_{\infty}+\frac{\delta}{2}\left[(1-\lambda)x^2+t^2\right]
\end{equation*}
is defined on $B_{1}$ and has a strict minimum in $(0,0)$. Notice that $h_{\delta}(0,0)=0$ and $h\geq0$ in $B_{1}$. We claim that 
\begin{equation}\label{gradient_positive}
 |\nabla h_{\delta}|>0\quad\mbox{a.e. in}\quad B_{1}\,.
\end{equation}
\textbf{Proof of (\ref{gradient_positive}).} Let $(x,t)\in B_{1}$ such that $h_{\delta}$ is differentiable in $(x,t)$ and $\nabla h_{\delta}(x,t)=0$. It follows since $g_{\infty}$ solves (\ref{HJ})
\begin{align*}
0&=\partial_t g_{\infty}+f(\partial_x g_{\infty})\\
&=\partial_t\psi(x_0,t_0)-\delta t+f(\lambda\partial_x\psi(x_0,t_0)+(1-\lambda)\delta x)\\
&\leq \partial_t\psi(x_0,t_0)+\lambda f(\partial_x\psi(x_0,t_0))+((1-\lambda)f(\delta x)-\delta t)\,.
\end{align*}
Since (\ref{psi_is_negative}) holds, we can choose $\delta$ and $\lambda$ small enough the expression
\[ \partial_t\psi(x_0,t_0)+\lambda f(\partial_x\psi(x_0,t_0))+\delta(f(\delta x)-t)\]
becomes strictly negative, which is a contradiction. Therefore the claim (\ref{gradient_positive}) is proved.

Further we choose $\delta$ and $\lambda$ small enough such that
\begin{equation}\label{choice_of_delta_lambda}
 |\partial_t\psi(x_0,t_0)|>\lambda\,\partial_x\psi(x_0,t_0)\cdot\sup_{s\in[-\|u\|_{\infty},\|u\|_{\infty}]}f'(s)+\delta((1-\lambda)x+t)\,.
\end{equation}
By $\tau>0$ we denote the minimum of $h_{\delta}$ on $\partial B_{1}$ and by $\overline{a}$ the essential supremum of $u_{\infty}$ on $\{h_{\delta}<\tau\}$. If $\overline{a}>0$ let $\underline{a}$ be close to $\overline{a}$ such that $0<\underline{a}<\overline{a}$. Let $A:=\{h_{\delta}<\tau\}\cap \{\underline{a}<u_{\infty}\}$. The set $A$ has positive Lebesgue measure. Therefore by the Coarea Formula and by $|\nabla h_{\delta}|>0$ it follows for $E_{s}:=\{h_{\delta}=s\}$
\begin{equation*}
 0<\int_{A}|\nabla h_{\delta}(x,t)|\,dx\,dt=\int_{0}^{\tau}\mathcal{H}^{1}(A\cap E_{s})\,ds\,.
\end{equation*}
Hence the set
\begin{equation*}
 S:=\left\{s\in(0,\tau): \mathcal{H}^{1}(\{\underline{a}<u_{\infty}\}\cap E_{s})>0,~\mathcal{H}^{1}(\{u_{\infty}>\overline{a}\}\cap E_{s})=0\right\}
\end{equation*}
has positive Lebesgue measure. For a vector $v=(v_{1},v_{2})$ we define $v^{\perp}:=(-v_{2},v_{1})$ and for a $s\in S$ the function
\begin{equation*}
 s\to l(s):=\int_{E_{s}}\left[X(u_{\infty}\wedge a)-\nabla^{\perp}\psi_{\infty}+\delta((1-\lambda)x,t)^{\perp}\right]\cdot\nu\,
\end{equation*}
where $\nu=\frac{\nabla h_{\delta}}{|\nabla h_{\delta}|}$ and the $X$ is the vectorfield from (\ref{vectorfield_X}). We choose $s\in S$ such that 
\begin{equation*}
 \lim_{\varepsilon\rightarrow0}\frac{1}{\varepsilon}\int^{s}_{s-\varepsilon}l(s')ds'=l(s)\,.
\end{equation*}
We define $\zeta_{\varepsilon}(x,t):=1\wedge(s-h_{\delta})^{+}/\varepsilon$ and calculate
\begin{equation*}
 \nabla \zeta_{\varepsilon}=\left\{\begin{array}{ll} 
					0&\mbox{if}~h_{\delta}>s~\mbox{or}~h_{\delta}<s-\varepsilon\\
                                   -\frac{1}{\varepsilon}\nabla h_{\delta} &\mbox{if}~ s-\varepsilon<h_{\delta}<s\,.
				\end{array}\right.
\end{equation*}
The choice of $s\in S$ an the Coarea Formula implies
\begin{multline*}
 \lim_{\varepsilon\rightarrow0}\int_{B_{1}}\left[X(u_{\infty}\wedge \underline{a})-\nabla^{\perp}\psi_{\infty}+\delta((1-\lambda)x,t)^{\perp}\right]\cdot\nabla\zeta_{\varepsilon}\\
=-\lim_{\varepsilon\rightarrow0}\frac{1}{\varepsilon}\int_{s-\varepsilon}^{s}l(s')ds'=l(s)\,.
\end{multline*}
The sign of $m_{\infty}$ gives
\begin{equation*}
 0\leq -\int_{B_{1}}\left[X(u_{\infty}\wedge \underline{a})-\nabla^{\perp}\psi_{\infty}+\delta((1-\lambda)x,t)^{\perp}\right]\cdot\nabla\zeta_{\varepsilon}\,.
\end{equation*}
As $\varepsilon\rightarrow0$ this implies
\begin{equation*}
 0\leq\int_{E_{s}}\left[X(u_{\infty}\wedge \underline{a})-\nabla^{\perp}\psi_{\infty}+\delta((1-\lambda)x,t)^{\perp}\right]\cdot\nu\,.
\end{equation*}
Now define $E^{+}_{s}:=E_{s}\cap\{u_{\infty}>\underline{a}\}$ and $E^{-}_{s}:=E_{s}\cap\{u_{\infty}\leq\underline{a}\}$. For $(x,t)\in E^{-}_{s}$ we notice
\begin{equation*}
 \left[X(u_{\infty}\wedge a)-\nabla^{\perp}\psi_{\infty}+\delta((1-\lambda)x,t)^{\perp}\right]\cdot\nu=\nabla^{\perp} h_{\delta}\cdot\nabla h_{\delta}=0\,.
\end{equation*}
Therefore it follows
\begin{equation*}
  0\leq\int_{E_{s}^+}\left[X(\underline{a})-\nabla^{\perp}\psi_{\infty}+\delta((1-\lambda)x,t)^{\perp}\right]\cdot\nabla h_{\delta}\,.
\end{equation*}
In order to get a contradiction we claim
\begin{equation}\label{negativity_of_the_whole_expression}
  \left(X(\underline{a})-\nabla^{\perp}\psi_{\infty}+\delta((1-\lambda)x,t)^{\perp}\right)\cdot\nabla h_{\delta}<0\,.
\end{equation}
We rearrange terms
\begin{multline}\label{rearanged_terms}
 \left[X(\underline{a})-\nabla^{\perp}\psi_{\infty}+\delta((1-\lambda)x,t)^{\perp}\right]\cdot\nabla h_{\delta}\\
=X(\underline{a})\cdot\nabla g_{\infty}+\left(\nabla\psi_{\infty}-\delta((1-\lambda)x,t)\right)\left(\nabla^{\perp}g_{\infty}-X(\underline{a})\right)\,.
\end{multline}
We show (\ref{negativity_of_the_whole_expression}), by proving that each term on the right hand side of (\ref{rearanged_terms}) is negative respectively strictly negative. Firstly we treat the first term and claim
\begin{equation}\label{negative_of_scalarproduct}
 X(\underline{a})\cdot\nabla g_{\infty}<0\,.
\end{equation}
A short calculation reveals
\begin{align*}
 X(\underline{a})\cdot\nabla g_{\infty}&=f(\underline{a})u_{\infty}-f(u_{\infty})\underline{a}\\
&=f(\underline{a})(u_{\infty}-\underline{a})+(f(a)-f(u_{\infty}))\underline{a}\\
&=\underline{a}(u_{\infty}-\underline{a})\left(\frac{f(\underline{a})-f(0)}{\underline{a}}-\frac{f(u_{\infty})-f(\underline{a})}{u_{\infty}-\underline{a}}\right)\,.
\end{align*}
By convexity of $f$ we have in the case $\underline{a}<u_{\infty}<\overline{a}<0$
\[\frac{f(u_{\infty})-f(\underline{a})}{u_{\infty}-\underline{a}}<\frac{f(\overline{a})-f(\underline{a})}{\overline{a}-\underline{a}}<\frac{f(\underline{a})-f(0)}{\underline{a}}\,.\]
This implies
\[\underline{a}\left(\frac{f(\underline{a})-f(0)}{\underline{a}}-\frac{f(u_{\infty})-f(\underline{a})}{u_{\infty}-\underline{a}}\right)\leq0\]
and henceforth (\ref{negative_of_scalarproduct}), if $\overline{a}\leq0$. On the other hand if $0<\underline{a}<\overline{a}$, we get for $\xi\in(0,\underline{a})$, $\alpha\in(\underline{a},u_{\infty})$
\[\frac{f(\underline{a})-f(0)}{a}=f'(\xi)<f'(\underline{a})<f'(\alpha)=\frac{f(u_{\infty})-f(\underline{a})}{u_{\infty}-\underline{a}}\,,\]
which implies (\ref{negative_of_scalarproduct}). Hence the first term of (\ref{rearanged_terms}) is non-positive and it remains to treat the second term. A short calculation gives
\begin{multline*}
\left(\nabla\psi_{\infty}-\delta((1-\lambda)x,t)\right)\left(\nabla^{\perp}g_{\infty}-X(\underline{a})\right)\\
=(u_{\infty}-\underline{a})\left[\partial_t\psi(x_0,t_0)+\lambda\partial_x\psi(x_0,t_0)\frac{f(u_{\infty})-f(\underline{a})}{u_{\infty}-\underline{a}}+\delta((1-\lambda)x+t)\right]\,.
\end{multline*}
Our choice of $\delta$ and $\lambda$ (see (\ref{choice_of_delta_lambda})) imply, that

\[\left(\nabla\psi_{\infty}-\delta((1-\lambda)x,t)\right)\left(\nabla^{\perp}g_{\infty}-X(\underline{a})\right)<0\]
and thus (\ref{negativity_of_the_whole_expression}). Finally (\ref{negativity_of_the_whole_expression}) implies
\begin{equation*}
\int_{E^{+}_{s}}\left(X(\underline{a})-\lambda\nabla^{\perp}\psi(x_{0},t_{0})+(1-\lambda)\delta(x,t)^{\perp}\right)\cdot\nabla h_{\delta}=0\,.
\end{equation*}
Since \[\left(X(\underline{a})-\lambda\nabla^{\perp}\psi(x_{0},t_{0})+(1-\lambda)\delta(x,t)^{\perp}\right)\cdot\nabla h_{\delta}<0\] it follows $\mathcal{H}^{1}(E^{+}_{s})=0$, which is a contradiction to our choice of $s\in S$. Thus
\begin{equation*}
 \partial_{t}\psi(x_{0},t_{0})+f\left(\partial_{x}\psi(x_{0},t_{0})\right)\geq0
\end{equation*}
as claimed. Henceforth $g$ is the viscosity solution of (\ref{HJ}) and $u=\partial_x g$ the entropy solution of (\ref{scl}) as claimed.
\end{proof}
\textbf{Proof of Theorem 1} Thanks to Lemma \ref{lemma_positiv_density} and Lemma \ref{proving_entropy} we can conclude the proof of Theorem \ref{minimality_theorem}. Indeed, we see that a weak solution $u\in L^{\infty}(\mathbb{R}\times[0,T))$ satisfying the assumptions of Theorem \ref{minimality_theorem}, has by Lemma \ref{lemma_positiv_density} ${\mathcal H}^1$-a.e. points of positive density, i.e.
\[\limsup_{r\rightarrow0^+}\frac1r\int_{B_{r}(x_0,t_0)}\int_{\mathbb{R}}m(x,t,a)\,da\,dx\,dt\geq0\quad\mbox{for }{\mathcal H}^1\mbox{ a.e. }\quad (x_0,t_0)\in\mathbb{R}\times(0,T)\,.\]
By Lemma \ref{proving_entropy} we know then, that $u$ has to be entropic.

%Literaturverzeichnis

\end{document}